\documentclass[oneside,french,titlepage]{amsart}

\usepackage[T1]{fontenc}
\usepackage[applemac]{inputenc}
\usepackage{array}
\usepackage{units}
\usepackage{mathrsfs}
\usepackage{multirow}
\usepackage{amstext}
\usepackage{amsthm}
\usepackage{amssymb}
\usepackage{stmaryrd}

\makeatletter

\providecommand{\tabularnewline}{\\}

\numberwithin{equation}{section}
\numberwithin{figure}{section}
\theoremstyle{plain}
\newtheorem{thm}{\protect\theoremname}
\theoremstyle{plain}
\newtheorem{prop}[thm]{\protect\propositionname}
\theoremstyle{definition}
\newtheorem{defn}[thm]{\protect\definitionname}
\theoremstyle{plain}
\newtheorem{cor}[thm]{\protect\corollaryname}
\theoremstyle{definition}
\newtheorem{example}[thm]{\protect\examplename}
\theoremstyle{remark}
\newtheorem{rem}[thm]{\protect\remarkname}
\theoremstyle{plain}
\newtheorem{lem}[thm]{\protect\lemmaname}

\usepackage[T1]{fontenc}

\usepackage{pgf,tikz} 
\usetikzlibrary{arrows}

\makeatother

\usepackage{babel}
\makeatletter
\addto\extrasfrench{%
   \providecommand{\og}{\leavevmode\flqq~}%
   \providecommand{\fg}{\ifdim\lastskip>\z@\unskip\fi~\frqq}%
}

\makeatother
\providecommand{\corollaryname}{Corollaire}
\providecommand{\definitionname}{Définition}
\providecommand{\examplename}{Exemple}
\providecommand{\lemmaname}{Lemme}
\providecommand{\propositionname}{Proposition}
\providecommand{\remarkname}{Remarque}
\providecommand{\theoremname}{Théorème}

\begin{document}
\begin{center}
\textbf{\Large{}Étude des opérateurs d'évolution en caractéristique
2}\bigskip{}
\par\end{center}

\begin{center}
\textbf{Richard Varro} \medskip{}
 
\par\end{center}

{\footnotesize{}Institut Montpelliérain Alexander Grothendieck,
Université de Montpellier, CNRS, Place Eugène Bataillon - 35095
Montpellier, France.}{\footnotesize\par}

\emph{\footnotesize{}E-mail}{\footnotesize{}: richard.varro@umontpellier.fr}{\footnotesize\par}

{\Large{}\bigskip{}
}{\Large\par}

\textbf{Abstract}: {\small{}We are interested in the evolution
operators defined on commutative and nonassociative algebras
when the scalar field is of characteristic 2. We distinguish
four types: nilpotent, quasi-constant, ultimately periodic and
plenary train operators. They are studied and classified for
non baric and for baric algebras.}{\small\par}

\medskip{}

\emph{Key words} : {\small{}Evolution operators, solvable algebras,
quasi-constant algebras, ultimately periodic operator, plenary
train algebras, evolution algebras, baric algebras, Bernstein
algebras, periodic Bernstein algebras.}{\small\par}

\emph{\small{}2010 MSC}{\small{} : Primary : 17D92, Secondary
: 17A30}.

\medskip{}

\section{Introduction}

Soit $\mathbb{K}$ un corps commutatif sans condition sur la
caractéristique. Etant donnée une $\mathbb{K}$-algèbre $A$,
l'\emph{opérateur d'évolution sur} $A$ est l'application $V:A\rightarrow A$
définie par $x\mapsto x^{2}$. \medskip{}

Cet opérateur joue un rôle important en algèbre génétique, car
il permet d'exprimer la distribution génétique $V\left(x\right)$
des descendants à partir de la distribution $x$ de la population
parentale (cf. \cite{Lyub-92}, p. 15; \cite{W-B-80}, p. 7).
En effet, soient $e_{1},\ldots,e_{n}$ des types génétiques autosomiques
présents dans une population panmictique et pangamique, si $\left(\alpha_{i}\left(t\right)\right)_{1\leq i\leq n}$,
$\sum_{i=1}^{n}\alpha_{i}\left(t\right)=1$ est la distribution
des fréquences de ces types à la génération $t$ et si $\gamma_{ijk}$
est la probabilité que l'union entre les types $e_{i}$ et $e_{j}$
donne un zygote de type $e_{k}$ qui atteint le stade de reproduction,
alors on a $\alpha_{k}\left(t+1\right)=\sum_{i,j=1}^{n}\gamma_{ijk}\alpha_{i}\left(t\right)\alpha_{j}\left(t\right)$.
On peut représenter algébriquement cette situation en donnant
un $\mathbb{K}$-espace vectoriel $A$ de base $\left(e_{1},\ldots,e_{n}\right)$
muni de la structure d'algèbre $e_{i}e_{j}=\sum_{k=1}^{n}\gamma_{ijk}e_{k}$,
pour $x\left(t\right)=\sum_{i=1}^{n}\alpha_{i}\left(t\right)e_{i}$
on a $V\left(x\left(t\right)\right)=\sum_{k=1}^{n}\bigl(\sum_{i,j=1}^{n}\gamma_{ijk}\alpha_{i}\left(t\right)\alpha_{j}\left(t\right)\bigr)e_{k}$
et donc $x\left(t+1\right)=V\left(x\left(t\right)\right)$.

\medskip{}

Les premiers travaux sur les applications quadratiques en lien
avec la Génétique datent de 1923. Dans \cite{Bernst-23(1),Bernst-23(2),Bernst-42},
S. N. Bernstein s'intéresse à la classification des opérateurs
d'évolution satisfaisant à ce qu'il appelle le \og principe de
stationnarité\fg . Génétiquement cela se traduit dans une population
par un équilibre de la distribution des fréquences d'un type
génétique après une génération. Mathématiquement cela se traduit
par la recherche des applications quadratiques $V:S\rightarrow S$
définies sur le simplexe standard $S=\left\{ \left(x_{1},\ldots,x_{n}\right)\in\mathbb{R}^{n}:x_{i}\geq0\text{ et }\sum x_{i}=1\right\} $
par $V\left(x\right)=\sum_{i,j=1}^{n}\gamma_{ijk}x_{i}x_{j}$
où $\gamma_{ijk}=\gamma_{jik}\geq0$ ($i,j,k=1,\ldots,n$) et
$\sum_{k=1}^{n}\gamma_{ijk}=1$ ($i,j=1,\ldots,n$) vérifiant
la condition $V^{2}=V$. A partir de 1971, Ju. I. Ljubi\v{c}
reprend l'étude du problème de Bernstein, dans (\cite{Lyub-74}
p. 594) il donne la première traduction algébrique de ce problème
sous la forme $\left(x^{2}\right)^{2}=s\left(x\right)^{2}x^{2}$
où $s$ est une pondération, la notion de pondération ayant été
introduite par I. M. H. Etherington \cite{Ether-39}. En 1975,
P. Holgate \cite{Holg-75} partant de cette interprétation algébrique
du problème de Bernstein obtient la décomposition de Peirce et
une classification en dimension 3 et 4 de ces algèbres appelées
depuis algèbres de Bernstein. En 1976, V. M. Abraham \cite{Abrah-76,Abrah-80}
donne une généralisation des algèbres de Bernstein: les algèbres
de Bernstein d'ordre $n$ qui sont étudiées par C. Mallol \cite{CM-89,MMO-91},
ces algèbres sont à leur tour généralisées en 1992, en algèbres
de Bernstein périodiques \cite{RV-92,RV-94}. Parallèlement,
A. A. Krapivin \cite{Krap-76} introduit une sous-classe des
algèbres de Bernstein: les algèbres quasi-constantes, étudiées
par I. Katambe \cite{Katamb-85,KKM-89}. Enfin les algèbres de
Bernstein et leurs généralisations appartiennent à une classe
plus vaste: les algèbres train plénières, étudiées par C. J.
Gutiérrez Fernández \cite{GutFer-00}. Dans tous ces travaux
le corps de base de ces algèbres est supposé de caractéristique
différente de 2. Des résultats en caractéristique 2 ont été donnés
dans \cite{CM-89} pour les algèbres de Bernstein de degré $n$,
dans \cite{B-M-95} pour les algèbres quasi-constantes de degré
$1$ et les train algèbres de degré 3 et dans \cite{RV-94} pour
les algèbres de Bernstein périodiques.

\bigskip{}

En ce qui concerne le plan de ce papier. Après avoir donné en
préliminaire des propriétés des opérateurs d'évolution en caractéristique
2 et définit la notion de semi-isomorphisme d'algèbres, on aborde
l'étude et la classification de quatre types d'opérateurs d'évolution:
nilpotents, quasi-constants, ultimement périodiques et train
pléniers. On montre que sauf pour le type nilpotent, l'existence
de ces opérateurs d'évolution sur des algèbres non pondérées
impose des conditions sur le cardinal du corps. On étudie ensuite
les opérateurs d'évolution des $\mathbb{F}_{2}$-algèbres d'évolution
de dimension finie. On termine par l'étude des opérateurs d'évolution
définis sur des algèbres pondérées.\bigskip{}

Tout au long de ce papier, $\mathbb{K}$ désigne un corps commutatif
sans condition sur la caractéristique, on note $\mathbb{F}$
quand le corps $\mathbb{K}$ est de caractéristique 2, enfin
on utilise la notation classique $\mathbb{F}_{2^{p}}$ pour le
corps de caractéristique 2 à $2^{p}$ éléments isomorphe à un
corps de décomposition sur $\nicefrac{\mathbb{Z}}{2\mathbb{Z}}$
du polynôme $X^{2^{p}}-X$. Les $\mathbb{K}$-algèbres et les
$\mathbb{F}$-algèbres considérées sont commutatives et ne sont
pas nécessairement associatives. L'ensemble $\mathbb{N}^{*}\times\mathbb{N}$
est muni de l'ordre lexicographique $\preccurlyeq$.

\medskip{}

\section{Préliminaire}

Nous donnons ici quelques propriétés des opérateurs d'évolution
en caractéristique 2 utilisées dans la suite.
\begin{prop}
Soit $A$ une $\mathbb{K}$-algèbre telle que $A^{2}\neq\left\{ 0\right\} $.
L'application $V:A\rightarrow A$, $x\mapsto x^{2}$ est additive
si et seulement si $\mbox{\emph{char}}\left(\mathbb{K}\right)=2$.
\end{prop}

\begin{proof}
La condition suffisante est immédiate. Pour la condition nécessaire,
comme $A^{2}\neq\left\{ 0\right\} $ il existe $x,y\in A$ tel
que $xy\neq0$, alors de $V\left(x+y\right)=V\left(x\right)+V\left(y\right)$
il résulte $2xy=0$, d'où le résultat.
\end{proof}
Soit $A$ une $\mathbb{K}$-algèbre, l'algèbre dérivée $A^{2}$
est la sous-algèbre de $A$ engendrée par l'ensemble $\left\{ xy;x,y\in A\right\} $,
quand le corps $\mathbb{K}$ est de caractéristique $\neq2$
et la $\mathbb{K}$-algèbre $A$ est commutative, l'algèbre $A^{2}$
est engendrée par l'ensemble $V\left(A\right)=\left\{ x^{2};x\in A\right\} $
qui n'est pas un sous-espace de $A^{2}$ .
\begin{prop}
\label{prop:EV-V^k(A)}Soit $A$ une $\mathbb{F}$-algèbre, si
le corps $\mathbb{F}$ est parfait alors quel que soit l'entier
$k\geq1$ l'ensemble $V^{k}\left(A\right)$ est un sous-espace
vectoriel de $A^{2}$.
\end{prop}

\begin{proof}
On a $V\left(A\right)\subset A^{2}\subset A$ on en déduit que
$V^{k}\left(A\right)\subset V\left(A\right)$ pour tout $k\geq1$.
L'application $V$ étant additive, pour tout entier $k\geq1$
l'ensemble $V^{k}\left(A\right)$ est un sous-groupe de $A^{2}$.
Soient $k\geq1$, $x\in A$ et $\lambda\in\mathbb{F}$, le corps
$\mathbb{F}$ étant parfait, le morphisme de Frobenius $\text{Frob}:\mathbb{F}\rightarrow\mathbb{F}$,
$\alpha\mapsto\alpha^{2}$ est surjectif, il existe donc $\mu\in\mathbb{F}$
tel que $\lambda=\text{Frob}^{k}\left(\mu\right)$ donc $\lambda=\mu^{2^{k}}$
et $\lambda V^{k}\left(x\right)=V^{k}\left(\mu x\right)$ d'où
$\lambda V^{k}\left(x\right)\in V^{k}\left(A\right)$.
\end{proof}
Pour la classification des opérateurs d'évolution en caractéristique
2 on utilisera la forme affaiblie d'isomorphisme suivante:
\begin{defn}
Soient $A_{1}$ et $A_{2}$ deux algèbres, on dit que: 

– $A_{1}$ et $A_{2}$ sont \emph{semi-isomorphes} s'il existe
un isomorphisme vectoriel $f:A_{1}\rightarrow A_{2}$ tel que
$f\left(x_{1}^{2}\right)=f\left(x_{1}\right)^{2}$ pour tout
$x_{1}\in A_{1}$.

– les opérateurs d'évolution $V_{i}:A_{i}\rightarrow A_{i}$,
sont semblables s'il existe un semi-isomorphisme $f:A_{1}\rightarrow A_{2}$
tel que $f\circ V_{1}=V_{2}\circ f$.
\end{defn}

Il résulte immédiatement de la définition que 
\begin{prop}
Deux opérateurs d'évolution sont semblables si et seulement si
les algèbres sur lesquelles ils opérent sont semi-isomorphes.
\end{prop}

\vspace{0mm}

\begin{prop}
\label{prop:Semi-iso=000026V^k(A)}Si $A_{1}$ et $A_{2}$ sont
deux $\mathbb{F}$-algèbres semi-isomorphes alors pour tout entier
$k\geq1$ les algèbres engendrées par $V^{k}\left(A_{1}\right)$
et $V^{k}\left(A_{2}\right)$ sont semi-isomorphes.
\end{prop}

\begin{proof}
Soit $f:A_{1}\rightarrow A_{2}$ un semi-isomorphisme, de $f\circ V=V\circ f$
on déduit que $f\circ V^{k}=V^{k}\circ f$ pour tout $k\geq1$,
par conséquent la restriction de $f$ à $V^{k}\left(A_{1}\right)$
a pour image $V^{k}\left(A_{2}\right)$ et pour $y\in V^{k}\left(A_{2}\right)$
il existe $x\in A_{1}$ tel que $y=V^{k}\left(x\right)$ alors
$f\left(y^{2}\right)=f\left(V^{k+1}\left(x\right)\right)=V\left(f\left(V^{k}\left(x\right)\right)\right)=f\left(y\right)^{2}$.
\end{proof}
La notion de semi-isomorphisme est utile pour la classification
des opérateurs d'évolution en caractéristique 2 admettant un
polynôme annulateur.
\begin{prop}
Soient $A$ une $\mathbb{F}$-algèbre et $P\in\mathbb{F}\left[X\right]$,
si $P\left(V\right)\left(x\right)=0$ pour tout $x\in A$ alors
pour toute algèbre $A'$ semi-isomorphe à $A$ on a $P\left(V\right)\left(x'\right)=0$
pour tout $x'\in A'$.
\end{prop}

\begin{proof}
Soient $P\left(X\right)=\sum_{k=0}^{n}\alpha_{k}X^{k}$ et $f:A\rightarrow A'$
un semi-isomorphisme. Pour tout $x'\in A'$ il existe $x\in A$
tel que $f\left(x\right)=x'$ et on a
\[
P\left(V\right)\left(x'\right)=\sum_{k=0}^{n}\alpha_{k}V^{k}\left(f\left(x\right)\right)=f\left(\sum_{k=0}^{n}\alpha_{k}V^{k}\left(x\right)\right)=0.
\]
\end{proof}
\medskip{}

\section{Opérateurs d'évolution nilpotents, quasi-constants, ultimement
périodiques et train pléniers}

\subsection{Opérateurs d'évolution nilpotents }

\textcompwordmark{}

\smallskip{}

Dans une $\mathbb{K}$-algèbre $A$, on définit les puissances
plénières d'un élément $x\in A$ par $x^{\left[1\right]}=x$
et $x^{\left[n+1\right]}=x^{\left[n\right]}x^{\left[n\right]}$
pour $n\geq1$. 
\begin{defn}
Soit $A$ une $\mathbb{K}$-algèbre. 

Un élément $x\in A$ est \emph{nil-plénier de degré} $n$ si
on a $x^{\left[n+1\right]}=0$ et $x^{\left[n\right]}\neq0$. 

L'algèbre $A$ est \emph{nil-plénière de degré} $n$ si pour
tout $x\in A$ on a $x^{\left[n+1\right]}=0$ et s'il existe
dans $A$ un élément nil-plénier de degré $n$.

L'opérateur d'évolution $V$ sur $A$ est \emph{nilpotent de
degré} $n$ si on a $n\geq1$ et $V^{n}=0$ et $V^{n-1}\neq0$.
\end{defn}

En remarquant que pour tout $x\in A$ et tout entier $n\geq1$
on a $V^{n}\left(x\right)=x^{\left[n+1\right]}$ on a immédiatement
que
\begin{prop}
Une $\mathbb{K}$-algèbre $A$ est nil-plénière de degré $n$
si et seulement si l'opérateur d'évolution $V$ sur $A$ est
nilpotent de degré $n$.
\end{prop}

De la définition d'une algèbre nil-plénière on déduit.
\begin{prop}
Une $\mathbb{F}$-algèbre $A$ est nil-plénière de degré $n$
si et seulement si il existe une base $\left(e_{i}\right)_{i\in I}$
de $A$ telle que $V^{n}\left(e_{i}\right)=0$ pour tout $i\in I$
et $V^{n-1}\left(e_{j}\right)\neq0$ pour au moins un $j\in I$.
\end{prop}

\begin{proof}
Pour la condition nécessaire, on a $A=\ker V^{n}$ et par définition
il existe $x\in A$ tel que $V^{n}\left(x\right)=0$ et $V^{n-1}\left(x\right)\neq0$,
il suffit de compléter $\left\{ x\right\} $ en une base de $A$
pour établir le résultat. La condition suffisante résulte de
l'additivité de $V$. 
\end{proof}
\begin{prop}
\label{prop:Base_Alg_Res}Une $\mathbb{F}$-algèbre de dimension
$d$, nil-plénière de degré $n$ est semi-isomorphe à une $\mathbb{F}$-algèbre,
notée $A\left(\mathbf{s}\right)$, définie par la donnée:

– d'un $M$-uplet d'entiers $\mathbf{s}=\left(s_{1},\ldots,s_{M}\right)$
tel que $M\geq1$, $n=s_{1}\geq\ldots,\geq s_{M}\geq1$ et $s_{1}+\cdots+s_{M}=d$; 

– d'une base $\bigcup_{i=1}^{M}\left\{ e_{i,j};1\leq j\leq s_{i}\right\} $
telle que
\[
e_{i,j}^{2}=\begin{cases}
e_{i,j+1} & \text{si }1\leq j\leq s_{i}-1\\
0 & \text{si }j=s_{i}
\end{cases}
\]
 et les autres produits étant définis arbitrairement.

Et deux algèbres $A\left(\mathbf{s}\right)$ et $A\left(\mathbf{t}\right)$
sont semi-isomorphes si et seulement si $\mathbf{s}=\mathbf{t}$.
\end{prop}

\begin{proof}
Soit $A$ est une $\mathbb{F}$-algèbre de dimension $d$, nil-plénière
de degré $n$. Par restriction du corps des scalaires à $\mathbb{F}_{2}$,
l'opérateur d'évolution est linéaire sur $A$ et vérifie $V^{n}=0$,
$V^{n-1}\neq0$. De la décomposition de Frobenius\footnote{(Cyclic decomposition theorem) }
de $V$ on obtient la décomposition de $A$ en sous-espaces cycliques
dont les bases vérifient les conditions de l'énoncé, de plus
la suite des invariants de similitude de $V$ étant $\left(X^{s_{1}},\ldots,X^{s_{M}}\right)$
on en déduit immédiatement que les espaces $A\left(\mathbf{s}\right)$
et $A\left(\mathbf{t}\right)$ sont isomorphes si et seulement
si on a $\mathbf{s}=\mathbf{t}$. Ces résultats sont conservés
par extension du corps des scalaires de $\mathbb{F}_{2}$ à $\mathbb{F}$.
\end{proof}
\begin{cor}
Si $A$ est une $\mathbb{F}$-algèbre nil-plénière de degré $n$
alors $\dim A\geq n$.
\end{cor}

\begin{prop}
\label{prop:dimA_Resol}Si $A$ est une $\mathbb{F}$-algèbre
nil-plénière de degré $n$ de dimension finie, alors il existe
une partie finie $\mathcal{F}$ de $A$ telle que
\[
A=\bigoplus_{k=1}^{n}\text{\emph{span}}\left(V^{k}\left(\mathcal{F}\right)\right).
\]
\end{prop}

\begin{proof}
Soit $A$ est une $\mathbb{F}$-algèbre de dimension finie nil-plénière
de degré $n$. Par définition de l'algèbre $A$, il existe un
élément $x_{1}\in A$ vérifiant $V^{n}\left(x_{1}\right)=0$
et $V^{n-1}\left(x_{1}\right)\neq0$, alors le système $\left\{ V^{k}\left(x_{1}\right);0\leq k\leq n-1\right\} $
est libre. Si le système $\mathscr{S}=\left\{ V^{j}\left(x_{i}\right);1\leq i\leq p,0\leq j<s_{i}<n\right\} $
est libre et si $\mathscr{S}$ n'est pas générateur de $A$,
alors il existe $x_{p+1}\in A$ et un entier $1\leq s_{p+1}<n$
tels que $V^{s_{p+1}}\left(x_{p+1}\right)=0$, $V^{s_{p+1}-1}\left(x_{p+1}\right)\neq0$
et $V^{s_{p+1}-1}\left(x_{p+1}\right)\notin\text{span}\left(\mathscr{S}\right)$,
de plus le système $\mathscr{S}\cup\left\{ V^{k}\left(x_{p+1}\right);0\leq k<s_{p+1}\right\} $
est libre, en effet si 
\[
\sum_{i=1}^{p+1}\sum_{j=1}^{s_{i}-1}\lambda_{ij}V^{j}\left(x_{i}\right)=0
\]
il en résulte que
\begin{align*}
0 & =V^{s_{p+1}-1}\left(\sum_{i=1}^{p+1}\sum_{j=0}^{s_{i}-1}\lambda_{ij}V^{j}\left(x_{i}\right)\right)\\
 & =\lambda_{p+1,s_{p+1}-1}V^{s_{p+1}-1}\left(x_{p+1}\right)+\sum_{i=1}^{p}\sum_{j=0}^{s_{i}-1}\lambda_{ij}V^{s_{p+1}+j-1}\left(x_{i}\right)
\end{align*}
d'où $\lambda_{p+1,s_{p+1}-1}=0$ et par récurrence on obtient
$\lambda_{p+1,j}=0$ ($0\leq j<s_{p+1}$). En poursuivant le
raisonnement on obtient une base $\left\{ V^{j}\left(x_{i}\right);1\leq i\leq m,0\leq j<s_{i}\right\} $
de $A$ et en prenant $\mathcal{F}=\left\{ x_{i};1\leq i\leq m\right\} $
on obtient que $A$ est engendré par $\bigcup_{k=1}^{n}V^{k}\left(\mathcal{F}\right)$,
d'où le résultat.
\end{proof}
\medskip{}

\subsection{Opérateurs d'évolution quasi-constants }

\textcompwordmark{}

\smallskip{}

\begin{defn}
Une $\mathbb{K}$-algèbre $A$ est \emph{quasi-constante de degré}
$n$ s'il existe $e\in A$, $e\neq0$ tel que $V^{n}\left(x\right)=e$
pour tout $x\in A$, $x\neq0$, avec $n\in\mathbb{N}^{*}$ minimal,
on désigne par $\left(A,e\right)$ une telle algèbre. Dans ce
cas l'opérateur d'évolution $V:A\rightarrow A$ est dit quasi-constant
de degré $n$.
\end{defn}

La définition d'un opérateur quasi-constant $V:A\rightarrow A$
dépend de la donnée d'un élément $e\in A$.
\begin{prop}
L'élément $e$ d'une $\mathbb{F}$-algèbre quasi-constante $\left(A,e\right)$
est unique et vérifie $e^{2}=e$.
\end{prop}

\begin{proof}
Soit $\left(A,e\right)$ une $\mathbb{F}$-algèbre quasi-constante
de degré $n$. Supposons qu'il existe un élément $e'\in A$ tel
que $\left(A,e'\right)$ soit quasi-constante de degré $n$,
de $V^{n}\left(x\right)=e'$ et $V^{n}\left(x\right)=e$ il vient
$e'=e$.

On a $e^{2}=V\left(e\right)=V\left(V^{n}\left(e\right)\right)=V^{n}\left(V\left(e\right)\right)$
et en posant $x=V\left(e\right)$ dans la définition on a finalement
$e^{2}=e$.
\end{proof}
\begin{thm}
\label{thm:Existence-QC}Pour qu'une $\mathbb{F}$-algèbre $A$
soit quasi-constante il faut que $\mathbb{F}=\mathbb{F}_{2}$. 
\end{thm}

\begin{proof}
Soit $\left(A,e\right)$ une $\mathbb{F}$-algèbre quasi-constante
de degré $n$. Pour $\lambda\in\mathbb{F}$, $\lambda\neq0$
on a $e=V^{n}\left(\lambda e\right)=\lambda^{2^{n}}e$ d'où $\lambda^{2^{n}}=1$
ce qui s'écrit aussi $\left(\lambda-1\right)^{2^{n}}=0$ par
conséquent $\mathbb{F}\setminus\left\{ 0\right\} =\left\{ 1\right\} $.
\end{proof}
A isomorphisme près il n'existe qu'une seule $\mathbb{F}_{2}$-algèbre
quasi constante.
\begin{prop}
Une $\mathbb{F}_{2}$-algèbre quasi-constante est de degré $1$,
de dimension 1 et isomorphe à l'algèbre $\mathbb{F}_{2}\bigl\langle e\bigr\rangle$
où $e^{2}=e$.
\end{prop}

\begin{proof}
Soit $\left(A,e\right)$ une $\mathbb{F}_{2}$-algèbre quasi-constante
de degré $n$. Supposons que $A$ soit de dimension $\geq2$,
il existe $x,y\in A$ linéairement indépendants donc $x+y\neq0$
et on a $V^{n}\left(x+y\right)=V^{n}\left(x\right)+V^{n}\left(y\right)=0\neq e$.
Par conséquent $A$ est de dimension 1 et comme $V\left(e\right)=e$
on a $n=1$.
\end{proof}
\medskip{}

\subsection{Opérateurs d'évolution ultimement périodiques}

\subsubsection{Définition et exemples}
\begin{defn}
Soit $A$ une $\mathbb{F}$-algèbre. On dit que:

– Un élément $x\in A$ est \emph{ultimement périodique de prépériode}
$n$ \emph{et de période} $p$, en abrégé $\left(n,p\right)$-périodique,
s'il vérifie $V^{n+p}\left(x\right)=V^{n}\left(x\right)$ avec
$\left(p,n\right)\in\mathbb{N}^{*}\times\mathbb{N}$ minimal
pour l'ordre lexicographique $\preccurlyeq$.

– L'algèbre $A$ et l'opérateur d'évolution $V:A\rightarrow A$
sont dits ultimement périodiques de prépériode $n$ et de période
$p$, en abrégé $\left(n,p\right)$-périodique, si on a $V^{n+p}\left(x\right)=V^{n}\left(x\right)$
pour tout $x\in A$ et s'il existe dans $A$ au moins un élément
$\left(n,p\right)$-périodique.

– Un élément $x\in A$, l'algèbre $A$ ou l'opérateur d'évolution
$V$ sont ultimement périodiques s'il existe deux entiers $n$
et $p$ tels que $x$ ou $A$ ou $V$ sont $\left(n,p\right)$-périodiques. 
\end{defn}

\begin{example}
Soient $\mathbb{F}=\mathbb{F}_{2}$ et $n,p>1$ deux entiers.
On définit la $\mathbb{F}_{2}$-algèbre $A$ de base $\left(a_{i}\right)_{1\leq i\leq p}\cup\left(b_{j}\right)_{1\leq j\leq n}$
par:
\[
a_{i}^{2}=\begin{cases}
a_{i+1} & 1\leq i<p\\
a_{1} & i=p
\end{cases},\qquad b_{j}^{2}=\begin{cases}
b_{j+1} & 1\leq j<n\\
0 & j=n
\end{cases},
\]
tous les autres produits étant nuls. On a $V^{p}\left(a_{i}\right)=a_{i}$,
$V^{n}\left(b_{j}\right)=0$ et $\lambda^{2^{k}}=\lambda$ pour
tout $k\geq0$, par conséquent pour $x=\sum_{i=1}^{p}\alpha_{i}a_{i}+\sum_{j=1}^{p}\beta_{j}b_{j}$
on a: 
\[
V^{n+p}\left(x\right)=\sum_{i=1}^{p}\alpha_{i}^{2^{n+p}}V^{n+p}\left(a_{i}\right)+\sum_{j=1}^{p}\beta_{j}^{2^{n+p}}V^{n+p}\left(b_{j}\right)=\sum_{i=1}^{p}\alpha_{i}V^{n}\left(a_{i}\right)=V^{n}\left(x\right).
\]
Il ne reste plus qu'à montrer que le couple $\left(p,n\right)$
est minimal. Pour cela il suffit de considérer $y=a_{1}+b_{1}$
et de remarquer que pour tout $0\leq m<n$ on a $V^{m}\left(y\right)=V^{m}\left(a_{1}\right)+b_{m+1}$
et $V^{n}\left(y\right)=V^{n}\left(a_{1}\right)$, puis pour
tout $1\leq q<p$ on a $V^{n+q}\left(y\right)=V^{n}\left(a_{q+1}\right)$,
tout ceci joint au fait que la restriction de $V$ au sous-espace
engendré par $\left(a_{i}\right)_{1\leq i\leq p}$ est bijective
établit que pour tout $0\leq i<j<n+p$ on a $V^{i}\left(y\right)\neq V^{j}\left(y\right)$. 
\end{example}

\smallskip{}

\begin{example}
\label{exa:Alg=0000E8bre-associ=0000E9e-=0000E0-SD}Algèbre associée
à un système dynamique fini. 

Un système dynamique fini $\left(E,f\right)$ est l'itération
d'une application $f:E\rightarrow E$ définie sur un ensemble
non vide fini $E$. 

Pour un système dynamique fini $\left(E,f\right)$ il existe
des entiers $n\geq1$ et $p\geq0$ tels que $f^{n+p}\left(x\right)=f^{n}\left(x\right)$
pour tout $x\in E$ et $\left(p,n\right)$ minimal pour l'ordre
$\preccurlyeq$. En effet, pour tout $x\in E$ l'ensemble $\left\{ f^{m}\left(x\right);m\geq1\right\} $
étant fini il existe des entiers $n_{x}<m_{x}$ tels que $f^{m_{x}}\left(x\right)=f^{n_{x}}\left(x\right)$
ainsi l'ensemble $\left\{ \left(m_{x}-n_{x},n_{x}\right);f^{m_{x}}\left(x\right)=f^{n_{x}}\left(x\right)\right\} $
n'est pas vide et possède un plus petit élément $\left(m_{x}-n_{x},n_{x}\right)$
pour l'ordre $\preccurlyeq$, en posant $p_{x}=m_{x}-n_{x}$
on a donc $f^{n_{x}+p_{x}}\left(x\right)=f^{n_{x}}\left(x\right)$.
Soit $m=\max\left\{ n_{x};x\in E\right\} $, pour tout $x\in E$
on a $f^{m+p_{x}}\left(x\right)=f^{m-n_{x}}f^{n_{x}+p_{x}}\left(x\right)=f^{m-n_{x}}f^{n_{x}}\left(x\right)=f^{m}\left(x\right)$.
Ensuite pour tout entier $k\geq1$ et tout $x\in E$ on a $f^{m+kp_{x}}\left(x\right)=f^{\left(k-1\right)p_{x}}f^{m+p_{x}}\left(x\right)=f^{m+\left(k-1\right)p_{x}}\left(x\right)$
d'où on déduit que $f^{m+kp_{x}}\left(x\right)=f^{m}\left(x\right)$,
et pour $q=\text{lcm}\left\{ p_{x};x\in E\right\} $ on a $f^{m+q}\left(x\right)=f^{m}\left(x\right)$
pour tout $x\in E$. Ainsi l'ensemble $\left\{ \left(q,m\right);f^{m+q}=f^{m}\right\} $
n'est pas vide, il admet un élément $\left(p,n\right)$ minimal
pour l'ordre lexicographique $\preccurlyeq$. 

\medskip{}

On peut associer une $\mathbb{F}_{2}$-algèbre ultimement périodique
à un système dynamique fini $\left(E,f\right)$: soient $E$
un ensemble de cardinal $N$ et $f:E\rightarrow E$ une application
telle que $f^{n+p}=f^{n}$ avec $\left(p,n\right)$ minimal pour
$\preccurlyeq$, On considère le $\mathbb{F}_{2}$-espace vectoriel
$A$ de base $\left(e_{i}\right)_{1\leq i\leq N}$ muni de la
structure algébrique: $e_{i}^{2}=e_{f\left(i\right)}$, ($1\leq i\leq N$)
les autres produits étant arbitraires. On a donc $V\left(e_{i}\right)=e_{f\left(i\right)}$
et $V^{n+p}\left(e_{i}\right)=e_{f^{n+p}\left(i\right)}=e_{f^{n}\left(i\right)}=V^{n}\left(e_{i}\right)$
donc $V^{n+p}\left(x\right)=V^{n}\left(x\right)$ pour tout $x\in A$.\smallskip{}

Par exemple, soient $E=\left\llbracket 0;12\right\rrbracket $
et $f:E\rightarrow E$ l'application qui à tout $x$ associe
$x^{2}+2$ modulo 13. On trouve $f\left(4\right)=f\left(9\right)=5$,
$f\left(5\right)=f\left(8\right)=1$, $f\left(1\right)=3$, $f\left(3\right)=f\left(10\right)=11$,
$f\left(11\right)=f\left(2\right)=6$, $f\left(6\right)=f\left(7\right)=12$,
$f\left(12\right)=3$, $f\left(3\right)=11$ et $f\left(7\right)=12$.
L'application $f$ vérifie $f^{8}=f^{4}$ donc la $\mathbb{F}_{2}$-algèbre
associée à $\left(E,f\right)$ est $\left(4,4\right)$-périodique.\smallskip{}

Autre exemple, soient $E=\mathbb{F}_{2^{2}}\cup\left\{ \infty\right\} $
et $f:E\rightarrow E$ l'application qui à $x$ associe $x+x^{-1}$
si $x\notin\left\{ 0,\infty\right\} $ et $\infty$ sinon. En
notant $\alpha$ une racine primitive de $X^{3}-1$ on a $\mathbb{F}_{2^{2}}=\left\{ 0,1,\alpha,\alpha^{2}\right\} $
et $f\left(\alpha\right)=f\left(\alpha^{2}\right)=1$, $f\left(1\right)=0$,
$f\left(0\right)=\infty$ et $f\left(\infty\right)=\infty$,
l'application $f$ et l'algèbre associée à $\left(E,f\right)$
sont $\left(1,3\right)$-périodique. 
\end{example}

\smallskip{}

\begin{example}
\label{exa:Automates}Automates cellulaires élémentaires sur
le réseau $\nicefrac{\mathbb{Z}}{n\mathbb{Z}}$, $\left(n\geq3\right)$.\smallskip{}

1) Automate cellulaire élémentaire XOR (règle 90, cf. \cite{Automata},
p. 224).

On considère la $\mathbb{F}_{2}$-algèbre $A$ de base $\left(e_{1},e_{2},\ldots,e_{n}\right)$
définie par: 
\[
e_{i}^{2}=\begin{cases}
e_{2}+e_{n} & \text{si }i=1\\
e_{i-1}+e_{i+1} & 2\leq i\leq n-1,\quad\text{et}\quad\\
e_{1}+e_{n-1} & \text{si }i=n
\end{cases}e_{i}e_{j}=0\text{ si }i\neq j.
\]

Pour faciliter les calculs on remarque que si $\sigma$ dénote
la permutation circulaire $\left(1,2,\ldots,n\right)$ on a $V\left(e_{i}\right)=e_{\sigma\left(i\right)}+e_{\sigma^{-1}\left(i\right)}$.
Ainsi on a $V^{2}\left(e_{i}\right)=e_{\sigma^{2}\left(i\right)}+e_{\sigma^{-2}\left(i\right)}$
et $V^{3}\left(e_{i}\right)=e_{\sigma^{3}\left(i\right)}+e_{\sigma^{-3}\left(i\right)}+e_{\sigma\left(i\right)}+e_{\sigma^{-1}\left(i\right)}$
par conséquent si $n=3$ on a $V^{3}\left(e_{i}\right)=V\left(e_{i}\right)$,
on en déduit que l'opérateur $V$ est $\left(1,2\right)$-périodique.
De manière analogue dans le cas $n=6$ on trouve $V^{6}\left(e_{i}\right)=V^{2}\left(e_{i}\right)$
donc $V$ est $\left(2,4\right)$-périodique et si $n=12$ l'opérateur
$V$ est $\left(4,8\right)$-périodique. 

On peut aussi remarquer que dans le cas $n=4$, l'opérateur $V$
est nilpotent de degré 4 et pour $n=8$ il est nilpotent de degré
8.

\smallskip{}

2) Automate cellulaire élémentaire de Fredkin (règle 150, cf.
\cite{Automata}, p. 224)).

On considère la $\mathbb{F}_{2}$-algèbre $A$ de base $\left(e_{1},e_{2},\ldots,e_{n}\right)$
munie de la structure: 
\[
e_{i}^{2}=\begin{cases}
e_{1}+e_{2}+e_{n} & \text{si }i=1\\
e_{i-1}+e_{i}+e_{i+1} & 2\leq i\leq n-1,\quad\text{et}\quad\\
e_{1}+e_{n}+e_{n-1} & \text{si }i=n
\end{cases}e_{i}e_{j}=0\text{ si }i\neq j.
\]

Comme ci-dessus en utilisant la permutation circulaire $\sigma=\left(1,2,\ldots,n\right)$
on a $V\left(e_{i}\right)=e_{\sigma\left(i\right)}+e_{i}+e_{\sigma^{-1}\left(i\right)}$.
Alors on a $V^{2}\left(e_{i}\right)=e_{\sigma^{2}\left(i\right)}+e_{i}+e_{\sigma^{-2}\left(i\right)}$
et $V^{3}\left(e_{i}\right)=e_{\sigma^{3}\left(i\right)}+e_{\sigma^{2}\left(i\right)}+e_{i}+e_{\sigma^{-2}\left(i\right)}+e_{\sigma^{-3}\left(i\right)}$
par conséquent dans le cas $n=3$ on a $V^{3}\left(e_{i}\right)=V^{2}\left(e_{i}\right)$
donc l'opérateur $V$ est $\left(2,1\right)$-périodique. De
la même manière on a pour $n=4$ que $V^{4}\left(e_{i}\right)=e_{i}$
donc $V$ est $\left(0,4\right)$-périodique, quand $n=6$ on
a $V$ qui est $\left(4,2\right)$-périodique et si $n=8$ on
trouve que $V$ est $\left(0,8\right)$-périodique. 

\medskip{}
\end{example}

\subsubsection{Propriétés des algèbres ultimement périodiques}
\begin{rem}
Quel que soit le corps $\mathbb{F}$, toute $\mathbb{F}$-algèbre
$A$ nil-plénière est ultimement périodique car il existe un
entier $n\geq1$ tel que $V^{n}\left(x\right)=0$ pour tout $x\in A$
et $V^{n-1}\left(x\right)\neq0$ pour un $x\in A$, donc $V^{n+1}\left(x\right)=V^{n}\left(x\right)$
pour tout $x\in A$ et pour un $x\in A$ tel que $V^{n-1}\left(x\right)\neq0$
on a $V^{n}\left(x\right)\neq V^{n-1}\left(x\right)$ donc $A$
est $\left(n,1\right)$-périodique. 

Toute $\mathbb{F}_{2}$-algèbre quasi-constante est ultimement
périodique. En effet, si $\left(A,e\right)$ est une algèbre
quasi-constante de degré $n$ on a $V^{n}\left(x\right)=e$ et
$e^{2}=e$ on déduit que $V^{n+1}\left(x\right)=V^{n}\left(x\right)$,
de la condition de minimalité du degré on déduit que $A$ est
$\left(n,1\right)$-périodique. 
\end{rem}

Désormais dans ce qui suit toutes les algèbres considérées sont
supposées non nil-plénières et non quasi-constantes. \medskip{}

Comme pour les algèbres quasi-constantes, l'ultime périodicité
d'une $\mathbb{F}$-algèbre impose des conditions sur le corps
$\mathbb{F}$.
\begin{thm}
\label{thm:Existence _des_UP} Pour qu'une $\mathbb{F}$-algèbre
$A$ soit $\left(n,p\right)$-périodique il faut que le corps
$\mathbb{F}$ vérifie $\mathbb{F}\subset\mathbb{F}_{2^{p}}$. 
\end{thm}

\begin{proof}
Etant donnée $A$ une $\mathbb{F}$-algèbre $\left(n,p\right)$-périodique.
L'algèbre $A$ n'étant pas nil-plénière il existe $x\in A$ tel
que $V^{n}\left(x\right)\neq0$, alors pour tout $\alpha\in\mathbb{F}$
on a $V^{n+p}\left(x\right)=V^{n}\left(x\right)$ et $V^{n+p}\left(\alpha x\right)=V^{n}\left(\alpha x\right)$
d'où $\bigl(\alpha^{2^{n+p}}-\alpha^{2^{n}}\bigr)V^{n}\left(x\right)=0$,
on a donc $\alpha^{2^{n+p}}-\alpha^{2^{n}}=0$ ou encore $\left(\alpha^{2^{p}}-\alpha\right)^{2^{n}}=0$
quel que soit $\alpha\in\mathbb{F}$, autrement dit tous les
éléments de $\mathbb{F}$ sont racine de $X^{2^{p}}-X$, ce qui
implique $\mathbb{F}\subset\mathbb{F}_{2^{p}}$.
\end{proof}
Désormais quand on parle de $\mathbb{F}$-algèbre $\left(n,p\right)$-périodique
on suppose que le corps $\mathbb{F}$ vérifie la condition nécessaire
du théorème \ref{thm:Existence _des_UP}.

\medskip{}

Dans certains cas, l'étude des algèbres sur un corps de caractéristique
2 se ramène à celle des algèbres ultimement périodiques. 
\begin{prop}
\label{prop:Dim_finie=000026UP} Si le corps $\mathbb{F}$ est
fini alors toute $\mathbb{F}$-algèbre de dimension finie est
ultimement périodique. 
\end{prop}

\begin{proof}
Soient $\mathbb{F}=\left\{ \lambda_{1},\ldots,\lambda_{m}\right\} $
et $A$ une $\mathbb{F}$-algèbre de base $\left(e_{1},\ldots,e_{d}\right)$.
Pour tout $1\leq i\leq d$ et tout $1\leq j\leq m$, le système
$\left\{ V^{k}\left(\lambda_{i}e_{j}\right);0\leq k\leq d\right\} $
est lié donc il existe $T_{ij}\in\mathbb{F}\left[X\right]$ de
degré au plus $d$ tel que $T_{ij}\left(V\right)\left(\lambda_{i}e_{j}\right)=0$.
Soit $P=\text{lcm}\left\{ T_{ij};1\leq i\leq d,1\leq j\leq m\right\} $,
par construction on a $P\left(V\right)\left(x\right)=0$ pour
tout $x\in A$. Alors l'ensemble $\left\{ P\in\mathbb{F}\left[X\right];P\left(V\right)=0\right\} $
étant un idéal de $\mathbb{F}\left[X\right]$ non réduit à $\left\{ 0\right\} $,
il admet un générateur $T$ unitaire de degré $D$. Pour chaque
entier $s\geq D$ on note $R_{s}\in\mathbb{F}\left[X\right]$
le reste de la division de $X^{s}$ par $T$, on a $\text{deg}\left(R_{s}\right)<D$
et l'ensemble $\left\{ R_{s};s\geq D\right\} $ est de cardinal
au plus \textbf{$m^{D}$}, il existe donc deux entiers $1\leq r<s$
tels que $R_{s}=R_{r}$, on en déduit $V^{s}=V^{r}$ soit $V^{r+\left(s-r\right)}=V^{r}$.
Par conséquent l'ensemble $\left\{ \left(i,j\right)\in\mathbb{N}\times\mathbb{N}^{*};V^{j+i}=V^{j}\right\} $
n'est pas vide, il admet un plus petit élément $\left(p,n\right)$
pour l'ordre lexicographique. Si $\left|\mathbb{F}\right|\leq2^{p}$
alors l'algèbre $A$ est $\left(n,p\right)$-périodique, si $\left|\mathbb{F}\right|>2^{p}$
il existe un plus petit entier $q\geq2$ tel que $2^{qp}\geq\left|\mathbb{F}\right|$,
comme $V^{n+qp}=V^{n}$ on en déduit que l'algèbre $A$ est $\left(n,qp\right)$-périodique.
\end{proof}
Nous allons établir des propriétés sur les algèbres ultimement
périodiques en commençant par les éléments ultimement périodiques. 
\begin{prop}
\label{prop:Elt_cycl}Soient $A$ une $\mathbb{F}$-algèbre et
$x\in A$ un élément $\left(n,p\right)$-périodique.

1) On a $V^{m+kp}\left(x\right)=V^{m}\left(x\right)$, pour tout
$m\geq n$ et $k\geq0$.

2) Pour tout $m\geq n$, il existe $0\leq q<p$ tel que $V^{m}\left(x\right)=V^{n+q}\left(x\right)$.

3) On a $V^{u}\left(x\right)=V^{v}\left(x\right)$ pour $u,v\geq n$
si et seulement si $u-v\equiv0\mod p$.

4) Pour tout $0\leq k\leq n$; $V^{k}\left(x\right)$ est $\left(n-k,p\right)$-périodique.

5) Pour tout $k\geq n$; $V^{k}\left(x\right)$ est $\left(0,p\right)$-périodique. 
\end{prop}

\begin{proof}
1) On a $V^{m+kp}=V^{m-n}V^{n+kp}$, puis par récurrence sur
$k\geq0$ on montre que $V^{n+kp}=V^{n}$ d'où la relation cherchée.
\smallskip{}

2) En effet, si $m-n<p$ on prend $q=m-n$. Si $m-n\geq p$,
il existe $k\geq1$ et $0\leq q<p$ tels que $m-n=kp+q$ alors
le résultat découle de $V^{m}=V^{n}V^{m-n}$ et de 1).\smallskip{}

3) La condition suffisante résulte de 1). Pour la condition nécessaire,
étant donnés $u,v\geq n$ des entiers, soient $u-n\equiv u'\mod p$
et $v-n\equiv v'\mod p$ avec $0\leq u',v'<p$ alors d'après
1) on a $V^{u}\left(x\right)=V^{n+u'}\left(x\right)$ et $V^{v}\left(x\right)=V^{n+v'}\left(x\right)$
et donc $V^{n+u'}\left(x\right)=V^{n+v'}\left(x\right)$. Si
$u'>v'$ de $V^{n+v'+\left(u'-v'\right)}\left(x\right)=V^{n+v'}\left(x\right)$,
de $0\leq u'-v'<p$ et de la minimalité de $\left(p,n\right)$
il vient $u'-v'=0$ et donc $u-v\equiv0\mod p$. On aboutit au
même résultat si $u'<v'$.\smallskip{}

4) Pour tout $0\leq k\leq n$ on a $V^{n+p}\left(x\right)=V^{n-k+p}\left(V^{k}\left(x\right)\right)$
et $V^{n}\left(x\right)=V^{n-k}\left(V^{k}\left(x\right)\right)$
par conséquent $V^{n-k+p}\left(V^{k}\left(x\right)\right)=V^{n-k}\left(V^{k}\left(x\right)\right)$,
montrons que le couple $\left(p,n-k\right)$ vérifiant cette
relation est minimal pour l'ordre $\preccurlyeq$. Soit $\left(q,m\right)\in\mathbb{N}^{*}\times\mathbb{N}$
tel que $\left(q,m\right)\preccurlyeq\left(p,n\right)$ et $V^{k}\left(x\right)$
soit $\left(m,q\right)$-périodique, de $V^{m+k+q}\left(x\right)=V^{m+k}\left(x\right)$
et de la minimalité de $\left(p,n\right)$ on a $q=p$ et $m+k=n$
donc $m=n-k$ et $V^{k}\left(x\right)$ est $\left(n-k,p\right)$-périodique.
\smallskip{}

5) En effet, pour tout $k\geq n$ on a: 
\[
V^{p}\left(V^{k}\left(x\right)\right)=V^{k-n}\left(V^{n+p}\left(x\right)\right)=V^{k-n}\left(V^{n}\left(x\right)\right)=V^{k}\left(x\right).
\]
Et s'il existe $q\leq p$ tel que $V^{q}\left(V^{k}\left(x\right)\right)=V^{k}\left(x\right)$
avec $\left(q,0\right)$ minimal, d'après le résultat 2) il existe
$0\leq s<p$ tel que $V^{k}\left(x\right)=V^{n+s}\left(x\right)$
on a donc $V^{n+s+q}\left(x\right)=V^{q+k}\left(x\right)=V^{k}\left(x\right)=V^{n+s}\left(x\right)$
et en composant ceci par $V^{p-s}$ on obtient $V^{n+p+q}\left(x\right)=V^{n+p}\left(x\right)$
d'où $V^{q}\left(V^{n}\left(x\right)\right)=V^{n}\left(x\right)$,
or $x$ étant par hypothèse $\left(n,p\right)$-périodique, d'après
le résultat 4) l'élément $V^{n}\left(x\right)$ est $\left(0,p\right)$-périodique
et par minimalité de $\left(p,0\right)$ on a $q=p$.
\end{proof}
Soit $A$ une $\mathbb{K}$-algèbre, on appelle \emph{orbite}
de $x\in A$ l'ensemble 
\[
\mathcal{O}\left(x\right)=\left\{ V^{k}\left(x\right);k\geq0\right\} 
\]
 et on note $\left|\mathcal{O}\left(x\right)\right|$ son cardinal.
\begin{prop}
\label{prop:card_Oq^k} Soit $x$ un élément $\left(n,p\right)$-périodique
d'une $\mathbb{F}$-algèbre, on a: 
\[
\left|\mathcal{O}\left(V^{k}\left(x\right)\right)\right|=\begin{cases}
n+p-k & \mbox{si }0\leq k\leq n\\
p & \mbox{si }n\leq k.
\end{cases}
\]
\end{prop}

\begin{proof}
Pour $0\leq k\leq n$, on déduit de $V^{n+p-k}\left(V^{k}\left(x\right)\right)=V^{n}\left(x\right)=V^{n-k}\left(V^{k}\left(x\right)\right)$
que $\mathcal{O}\left(V^{k}\left(x\right)\right)=\left\{ V^{m}\left(V^{k}\left(x\right)\right);0\leq m<n+p-k\right\} $
il en résulte que $\left|\mathcal{O}\left(V^{k}\left(x\right)\right)\right|\leq n+p-k$.
Montrons par l'absurde que pour tout $\left(i,j\right)$ tel
que $0\leq i<j<n+p-k$ on a $V^{i}\left(V^{k}\left(x\right)\right)\neq V^{j}\left(V^{k}\left(x\right)\right)$.
S'il existe $0\leq i<j<n+p-k$ tels que $V^{i}\left(V^{k}\left(x\right)\right)=V^{j}\left(V^{k}\left(x\right)\right)\;\left(\star\right)$.
On a deux cas :

– Si $j<p$, alors $0<p+i-j<p$ et en composant $\left(\star\right)$
par $V^{n+p-k-j}$ on trouve $V^{n+p+i-j}\left(x\right)=V^{n}\left(x\right)$
ce qui est en contradiction avec la minimalité du couple $\left(p,n\right)$.

– Si $j\geq p$, alors $0\leq j-p<n-k$ et en composant $\left(\star\right)$
par $V^{n-\left(j-p\right)-k}$ on obtient $V^{n+p+i-j}\left(x\right)=V^{n}\left(x\right)$
avec $p+i-j<p$ ce qui est à nouveau en contradiction avec la
minimalité du couple $\left(p,n\right)$.

On a établi que $\left|\mathcal{O}\left(V^{k}\left(x\right)\right)\right|=n+p-k$.

Quand $k>n$, d'après le 2) de la proposition \ref{prop:Elt_cycl},
pour tout entier $m\geq0$ il existe $0\leq s<p$ tel que $V^{m+k}\left(x\right)=V^{n+s}\left(x\right)$
il en résulte que $\mathcal{O}\left(V^{k}\left(x\right)\right)=\left\{ V^{n+s}\left(x\right);0\leq s<p\right\} $
donc $\left|\mathcal{O}\left(V^{k}\left(x\right)\right)\right|\leq p$.
Montrons par l'absurde que pour tout $0\leq i<j<p$ on a $V^{n+i}\left(x\right)\neq V^{n+j}\left(x\right)$.
S'il existe $0\leq i<j<p$ tels que $V^{n+i}\left(x\right)=V^{n+j}\left(x\right)$,
en composant ceci par $V^{p-j}$ il vient $V^{n+p+i-j}\left(x\right)=V^{n}\left(x\right)$
avec $p+i-j<p$ ce qui est en contradiction avec la minimalité
du couple $\left(p,n\right)$.
\end{proof}
\begin{cor}
\label{coroll:Egalit_des_O}Soit $x$ un élément $\left(n,p\right)$-périodique,
quel que soit $k\geq n$ on a: $\mathcal{O}\left(V^{k}\left(x\right)\right)=\mathcal{O}\left(V^{n}\left(x\right)\right)$.
\end{cor}

\begin{proof}
Il est clair que si $k\geq n$ on a $\mathcal{O}\left(V^{k}\left(x\right)\right)\subset\mathcal{O}\left(V^{n}\left(x\right)\right)$
et comme d'après la proposition ci-dessus on a $\left|\mathcal{O}\left(V^{k}\left(x\right)\right)\right|=\left|\mathcal{O}\left(V^{n}\left(x\right)\right)\right|=p$
il en découle que $\mathcal{O}\left(V^{k}\left(x\right)\right)=\mathcal{O}\left(V^{n}\left(x\right)\right)$.
\end{proof}
On peut caractériser de plusieurs façons les éléments ultimement
périodiques.
\begin{prop}
\label{prop:Caracteris_cycl}Soient $A$ une $\mathbb{F}$-algèbre
et $x\in A$, les énoncés suivants sont équivalents:

(i) $x$ est $\left(n,p\right)$-périodique.

(ii) $V^{i}\left(x\right)\neq V^{j}\left(x\right)$ pour tout
$0\leq i<j<n+p$ et $V^{n+p}\left(x\right)=V^{n}\left(x\right)$.

(iii) $V^{n+i}\left(x\right)\neq V^{n}\left(x\right)$ pour tout
$0<i<p$, $V^{j+p}\left(x\right)\neq V^{j}\left(x\right)$ pour
tout $0<j<n$ et $V^{n+p}\left(x\right)=V^{n}\left(x\right)$.

(iv) $V^{n}\left(x\right)$ est $\left(0,p\right)$-périodique
et $V^{n-1}\left(x\right)$ est $\left(1,p\right)$-périodique.

(v) $\left|\mathcal{O}\left(x\right)\right|=n+p$ et $\left|\mathcal{O}\left(V^{n}\left(x\right)\right)\right|=\left|\mathcal{O}\left(V^{n+1}\left(x\right)\right)\right|=p$.
\end{prop}

\begin{proof}
$\left(i\right)\Rightarrow\left(ii\right)$ Supposons qu'il existe
$0\leq i<j<n+p$ tels que $V^{i}\left(x\right)=V^{j}\left(x\right)$,
en composant ceci par $V^{n+p-j}$ on obtient $V^{n+p+i-j}\left(x\right)=V^{n+p}\left(x\right)=V^{n}\left(x\right)$
avec $p+\left(i-j\right)<p$, contradiction.\smallskip{}

$\left(ii\right)\Rightarrow\left(iii\right)$ est immédiat.\smallskip{}

$\left(iii\right)\Rightarrow\left(i\right)$ \emph{Ad absurdum}.
Supposons qu'il existe $\left(d,m\right)\prec\left(p,n\right)$
tel que $V^{m+d}\left(x\right)=V^{m}\left(x\right)\;\left(\ast\right)$.
Si $d<p$, on ne peut pas avoir $m\leq n$ car en composant $\left(\ast\right)$
par $V^{n-m}$ on obtient $V^{n+d}\left(x\right)=V^{n}\left(x\right)$
ce qui est en contradiction avec les hypothèses, on a donc $m>n$
et dans ce cas il existe $k\geq0$ et $0\leq r<p$ tels que $m-n=kp+r$
d'où $V^{m+d}\left(x\right)=V^{n+r+d}\left(x\right)$ et $V^{m}\left(x\right)=V^{n+r}\left(x\right)$
et en composant $V^{n+r+d}\left(x\right)=V^{n+r}\left(x\right)$
par $V^{p-r}$ on obtient $V^{n+d}\left(x\right)=V^{n}\left(x\right)$
avec $d<p$, contradiction. Si $d=p$, par hypothèse on ne peut
pas avoir $\left(p,m\right)\prec\left(p,n\right)$ et $V^{m+p}\left(x\right)=V^{m}\left(x\right)$.
\smallskip{}

$\left(i\right)\Rightarrow\left(iv\right)$ De $V^{p}\left(V^{n}\left(x\right)\right)=V^{n}\left(x\right)$,
$V^{p+1}\left(V^{n-1}\left(x\right)\right)=V^{n-1}\left(x\right)$
et de la minimalité de $\left(p,n\right)$ on déduit que $V^{n}\left(x\right)$
est $\left(0,p\right)$-périodique et que $V^{n-1}\left(x\right)$
est $\left(1,p\right)$-périodique.\smallskip{}

$\left(iv\right)\Rightarrow\left(v\right)$ Puisque $V^{n}\left(x\right)$
est $\left(0,p\right)$-périodique, on a $V^{p}\left(V^{n}\left(x\right)\right)=V^{n}\left(x\right)$
par conséquent $V^{n+p}\left(x\right)=V^{n}\left(x\right)$.
Montrons par l'absurde que $x$ est $\left(n,p\right)$-périodique.
On suppose que $x$ est $\left(m,d\right)$-périodique avec $\left(d,m\right)\prec\left(p,n\right)$,
d'après le 4) de la proposition \ref{prop:Elt_cycl} on a que
$V^{m}\left(x\right)$ est $\left(0,d\right)$-périodique. Si
$m>n$, d'après le 5) de la proposition \ref{prop:Elt_cycl},
$V^{m}\left(x\right)$ est $\left(0,p\right)$-périodique, par
conséquent on a $d=p$ et donc $\left(p,m\right)\succ\left(p,n\right)$,
contradiction. Par conséquent on a $m\leq n$, alors du 5) de
la proposition \ref{prop:Elt_cycl} on a que $V^{n}\left(x\right)$
est $\left(0,d\right)$-périodique et de la minimalité de $\left(0,p\right)$
il vient que $d=p$, on ne peut pas avoir $m<n$ car sinon du
5) de la proposition \ref{prop:Elt_cycl} on aurait que $V^{n-1}\left(x\right)$
est $\left(0,p\right)$-périodique, par conséquent on a $m=n$.
Et on obtient les résultats de l'énoncé en appliquant la proposition
\ref{prop:card_Oq^k}.\smallskip{}

$\left(v\right)\Rightarrow\left(i\right)$ Montrons par l'absurde
que $V^{n+p}\left(x\right)=V^{n}\left(x\right)$ avec $\left(p,n\right)$
minimal. En effet s'il existe $0\leq i<j\leq p$ tels que $V^{n+i}\left(x\right)=V^{n+j}\left(x\right)$
on a $j=p$ car dans le cas où $j<p$ on aurait $\left|\mathcal{O}\left(V^{n}\left(x\right)\right)\right|\leq j<p$.
On a aussi $i=0$ car si $i>0$, de $V^{n+i}\left(x\right)=V^{n+p}\left(x\right)$
il vient $V^{n+1+i}\left(x\right)=V^{n+1+p}\left(x\right)$ on
aurait donc $\mathcal{O}\left(V^{n+1}\left(x\right)\right)\subset\left\{ V^{n+1+m}\left(x\right);i\leq m\leq p\right\} $
d'où $\left|\mathcal{O}\left(V^{n+1}\left(x\right)\right)\right|<p$.
Montrons que le couple $\left(p,n\right)$ vérifiant $V^{n+p}\left(x\right)=V^{n}\left(x\right)$
est minimal. S'il existait $\left(d,m\right)\in\mathbb{N}^{*}\times\mathbb{N}$
tel que $\left(d,m\right)\preccurlyeq\left(p,n\right)$ et $x$
est $\left(d,m\right)$-périodique, dans le cas où $m>n$ d'après
la proposition \ref{prop:card_Oq^k} on aurait $\left|\mathcal{O}\left(V^{n}\left(x\right)\right)\right|=d+m-n$
et $\left|\mathcal{O}\left(V^{n+1}\left(x\right)\right)\right|=d+m-n-1$
donc $\left|\mathcal{O}\left(V^{n}\left(x\right)\right)\right|\neq\left|\mathcal{O}\left(V^{n+1}\left(x\right)\right)\right|$,
contradiction. Par conséquent on a $m\leq n$ donc $\left|\mathcal{O}\left(V^{n}\left(x\right)\right)\right|=d$
d'où $d=p$ et $\left|\mathcal{O}\left(x\right)\right|=p+m$
d'où $m=n$.
\end{proof}
Concernant les éléments d'une algèbre ultimement périodique on
a les propriétés suivantes.
\begin{prop}
\label{prop:Proprietes_Alg_period} Soit $A$ une $\mathbb{F}$-algèbre
$\left(n,p\right)$-périodique, on a:

1) Pour tout $x\in A$, il existe $0\leq m\leq n$ et un diviseur
$d$ de $p$ tels que $x$ soit $\left(m,d\right)$-périodique.

2) Si $x\in A$ est $\left(m,d\right)$-périodique alors $0\leq m\leq n$
et $d$ un diviseur de $p$.

3) Pour tout entier $0\leq m\leq n$ et tout diviseur $d$ de
$p$, il existe $x\in A$ tel que $V^{m+d}\left(x\right)=V^{m}\left(x\right)$.

4) Un élément $x\in A$ est $\left(m,d\right)$-périodique si
et seulement si il existe un unique $\left(a,b\right)\in A\times A$
tel que $x=a+b$ avec $V^{d}\left(a\right)=a$, $V^{m}\left(b\right)=0$
et $\left|\mathcal{O}\left(a\right)\right|=d$, $\left|\mathcal{O}\left(b\right)\right|=m$.

5) Si $x',x''\in A$ sont respectivement $\left(n',p'\right)$-périodique
et $\left(n'',p''\right)$-périodique et si $x'+x''\neq0$ alors
$x'+x''$ est $\left(\mu,\delta\right)$-périodique avec $\mu,\delta$
vérifiant $\mu\leq\max\left(n',n''\right)$ et $\delta=\text{lcm}\left(p',p''\right)$
si $p'\neq p''$ ou $\delta$ diviseur de $p'$ quand $p'=p''$.

6) Il existe un entier $q$ multiple de $p$ tel que $V^{2q}\left(x\right)=V^{q}\left(x\right)$
pour tout $x\in A$. 
\end{prop}

\begin{proof}
1) Soit $x\in A$, d'après la proposition \ref{prop:card_Oq^k},
pour tout $j\ge0$ les ensembles $\mathcal{O}\left(V^{j}\left(x\right)\right)$
sont finis, posons $d=\min\left\{ \left|\mathcal{O}\left(V^{j}\left(x\right)\right)\right|;j\geq0\right\} $
et $m=\left|\mathcal{O}\left(x\right)\right|-d$, on a donc $\left|\mathcal{O}\left(x\right)\right|=m+d$.
Soit $k\geq0$ le plus petit élément de l'ensemble $\left\{ j;\left|\mathcal{O}\left(V^{j}\left(x\right)\right)\right|=d\right\} $,
on a $V^{d+k}\left(x\right)=V^{k}\left(x\right)$, alors $\left|\mathcal{O}\left(x\right)\right|=d+k$
d'où $k=m$ et la relation $V^{m+d}\left(x\right)=V^{m}\left(x\right)\;\left(\ddagger\right)$. 

Par construction des entiers $d$ et $m$, le couple $\left(d,m\right)$
vérifiant la relation $\left(\ddagger\right)$ est minimal pour
l'ordre lexicographique. Montrons que $m\leq n$ et que $d$
divise $p$. L'algèbre $A$ étant $\left(n,p\right)$-périodique
on a $V^{n+p}\left(x\right)=V^{n}\left(x\right)$ par conséquent
$1\leq d\leq p$ et $m<m+d\leq n+p$. En composant $\left(\ddagger\right)$
par $V^{n+p-m}$ on obtient $V^{n+p+d}\left(x\right)=V^{n+p}\left(x\right)$
d'où $V^{n+d}\left(x\right)=V^{n}\left(x\right)$, on en déduit
que $V^{n+qd}\left(x\right)=V^{n}\left(x\right)$ pour tout $q\geq1$.
Il existe deux entiers $q$ et $d'$ tels que $p=qd+d'$ avec
$0\leq d'<d$ alors $V^{n+p}\left(x\right)=V^{n+qd+d'}\left(x\right)=V^{n+d'}\left(x\right)$
donc $V^{n+d'}\left(x\right)=V^{n}\left(x\right)$ et par définition
de $d$ on en déduit que $d'=0$, donc $d$ divise $p$. Ensuite
en composant $\left(\ddagger\right)$ par $V^{\left(q-1\right)d}$
on trouve $V^{m+p}\left(x\right)=V^{m}\left(x\right)$ et par
minimalité du couple $\left(p,n\right)$ on a $m\leq n$.\smallskip{}

2) D'après 1) il existe $\left(m',d'\right)$ tel que $x$ est
$\left(m',d'\right)$-périodique avec $0\leq m'\leq n$ et $d'$
divisant $p$, alors par minimalité de $\left(d,m\right)$ on
a $\left(m,d\right)=\left(m',d'\right)$. \smallskip{}

3) Par définition d'une algèbre $\left(n,p\right)$-périodique
il existe $y\in A$ qui est $\left(n,p\right)$-périodique. Soit
$q\geq1$ un entier tel que $p=qd$, on pose $x=\sum_{k=0}^{q-1}V^{n-m+kd}\left(y\right)$
on a $V^{m+d}\left(x\right)=\sum_{k=1}^{q}V^{n+kd}\left(y\right)=V^{n}\left(y\right)+\sum_{k=1}^{q-1}V^{n+kd}\left(y\right)=V^{m}\left(x\right)$. 

\smallskip{}

4) La condition suffisante est immédiate. Pour la condition nécessaire,
par restriction des scalaires sur le $\mathbb{F}_{2}$-espace
vectoriel $A$ on a $\ker V^{m}\left(V^{d}-id\right)=\ker V^{m}\oplus\ker\left(V^{d}-id\right)$,
par conséquent si $x$ est $\left(m,d\right)$-périodique on
a $x\in\ker V^{m}\left(V^{d}-id\right)$ il existe donc $a\in\ker\left(V^{d}-id\right)$,
$b\in\ker V^{m}$ tels que $x=a+b$ avec de plus $V^{d-1}\left(a\right)\neq a$,
$V^{m-1}\left(b\right)\neq0$.\smallskip{}

5) Soit $m=\max\left(n',n''\right)$ et $d=\mbox{lcm}\left(p',p''\right)$
d'après le 1) de la proposition \ref{prop:Elt_cycl} on a $V^{m+d}\left(x'\right)=V^{m}\left(x'\right)$
et $V^{m+d}\left(x''\right)=V^{m}\left(x''\right)$ donc $V^{m+d}\left(x'+x''\right)=V^{m}\left(x'+x''\right)$. 

D'après le résultat 1) il existe deux entiers $\mu$, $\delta$
vérifiant $0\leq\mu\leq n$, $\delta$ divise $p$ et $x'+x''$
est $\left(\mu,\delta\right)$-périodique, on a $\left(\delta,\mu\right)\preccurlyeq\left(d,m\right)$.
Montrons que $\mu\leq m$ et $\delta$ divise $d$. Si $\delta=d$
on a aussitôt $\mu\leq m$. Etudions le cas $\delta<d$, d'après
la propriété 4) on a $x'=a'+b'$ et $x''=a''+b''$ avec $V^{m}\left(b'\right)=V^{m}\left(b''\right)=0$
donc $V^{m}\left(b'+b''\right)=0$, comme de plus $V^{\mu}\left(b'+b''\right)=0$
et $V^{\mu-1}\left(b'+b''\right)\neq0$ on a $\mu\leq m$. Soit
$d=\delta q+r$ avec $0\leq r<\delta$, de $V^{m}\left(x'+x''\right)=V^{m+d}\left(x'+x''\right)=V^{m-\mu}V^{\mu+\delta q+r}\left(x'+x''\right)=V^{m+r}\left(x'+x''\right)$
et de $r<\delta$ on déduit que $r=0$ et donc $\delta$ divise
$d$. 

Montrons que si $p'\neq p''$ on a $\delta=d$. Considérons le
cas $p'<p''$ et supposons que $\delta\neq d$. Posons $g=\text{\ensuremath{\gcd}}\left(p',p''\right)$
et $p'=gq'$, $p''=gq''$ avec $\text{\ensuremath{\gcd}}\left(q',q''\right)=1$,
on a $d=gq'q''=p'q''=q'p''$, $\text{\ensuremath{\gcd}}\left(p',q''\right)=\text{\ensuremath{\gcd}}\left(q',p''\right)=1$
par conséquent de $\delta$ divise $d$ on déduit que $\delta$
divise $p'$ ou $p''$. Si $\delta$ divise $p''$ de $V^{\delta+\mu}\left(x'+x''\right)=V^{\mu}\left(x'+x''\right)\;\left(\star\right)$
et du résultat 1) de la proposition \ref{prop:Elt_cycl} on a
$V^{p''+m}\left(x'+x''\right)=V^{m}\left(x'+x''\right)$, or
on a $V^{p''+m}\left(x''\right)=V^{m}\left(x''\right)$ par conséquent
$V^{p''+m}\left(x'\right)=V^{m}\left(x'\right)$, et d'après
le résultat 4) de la proposition \ref{prop:Elt_cycl}, ceci implique
que $p'$ divise $p''$, il en résulte que $d=p'$ et par suite
que $\delta$ divise $p'$. Alors de la relation $\left(\star\right)$
et du résultat 1) de la proposition \ref{prop:Elt_cycl}, il
vient $V^{p'+m}\left(x'+x''\right)=V^{m}\left(x'+x''\right)$,
ceci joint à $V^{p'+m}\left(x'\right)=V^{m}\left(x'\right)$
implique que $V^{p'+m}\left(x''\right)=V^{m}\left(x''\right)$
avec $p'<p''$, contradiction.\smallskip{}

6) Soit $k\geq1$ le plus petit entier tel que $kp>n$. On pose
$q=kp$, avec le 1) de la proposition \ref{prop:Elt_cycl} on
a $V^{q}=V^{q-n}V^{n}=V^{q-n}V^{n+kp}=V^{q-n}V^{n+q}=V^{2q}$.
\end{proof}
Illustrons avec un exemple le résultat 5) de la proposition précédente
dans le cas $p'=p''$.
\begin{example}
Soit $A$ la $\mathbb{F}_{2}$-algèbre commutative de base $\left(e_{1},e_{2},e_{3},e_{4}\right)$
définie par $V\left(e_{i}\right)=e_{i+1}$ pour $i=1,2,3$ et
$V\left(e_{4}\right)=e_{1}$, les produits non mentionnés étant
arbitraires. L'algèbre $A$ est $\left(0,4\right)$-périodique.

a) Les éléments $x'=e_{1}$, $x''=e_{2}$ sont $\left(0,4\right)$-périodiques
et il en est de même pour $x'+x''=e_{1}+e_{2}$.

b) Les éléments $x'=e_{1}$, $x''=e_{3}$ sont $\left(0,4\right)$-périodiques,
quant à $x'+x''=e_{1}+e_{3}$ il est $\left(0,2\right)$-périodique.

c) Les éléments $x'=e_{1}+e_{2}$, $x''=e_{3}+e_{4}$ sont $\left(0,4\right)$-périodiques,
l'élément $x'+x''=e_{1}+e_{2}+e_{3}+e_{4}$ est idempotent donc
$\left(0,1\right)$-périodique.
\end{example}

Concernant les algèbres ultimement périodiques on a les propriétés
suivantes.
\begin{prop}
\label{prop:base_UP}Les énoncés suivants sont équivalents:

(i) $A$ est $\left(n,p\right)$-périodique;

(ii) il existe une base $\left(e_{i}\right)_{i\in I}$ de $A$
telle que $V^{n+p}\left(e_{i}\right)=V^{n}\left(e_{i}\right)$
pour tout $i\in I$ et $\left|\mathcal{O}\left(e_{i}\right)\right|=n+p$
pour au moins un $i\in I$;

(iii) il existe une base $\left(a_{i}\right)_{i\in I}\cup\left(b_{j}\right)_{j\in J}$
de $A$ vérifiant les quatre conditions: $V^{n}\left(a_{i}\right)=0$,
$V^{p}\left(b_{j}\right)=b_{j}$, pour tout $\left(i,j\right)\in I\times J$
et $\left|\mathcal{O}\left(a_{i}\right)\right|=n$, $\left|\mathcal{O}\left(b_{j}\right)\right|=p$
pour au moins un $\left(i,j\right)\in I\times J$. 
\end{prop}

\begin{proof}
$\left(i\right)\Rightarrow\left(ii\right)$ D'après la proposition
\ref{prop:Caracteris_cycl} il existe $x\in A$ qui est $\left(n,p\right)$-périodique
donc $\left|\mathcal{O}\left(x\right)\right|=n+p$ , par restriction
du corps des scalaires à $\mathbb{F}_{2}$, l'opérateur $V$
est linéaire et $A=\ker\left(V^{n+p}-V^{n}\right)$, il suffit
alors de compléter $\left\{ x\right\} $ en une une base de $\ker\left(V^{n+p}-V^{n}\right)$.\smallskip{}

$\left(ii\right)\Rightarrow\left(i\right)$ Soit $x=\sum_{i\in I}\alpha_{i}e_{i}$,
on a $V^{n+p}\left(x\right)=\sum_{i\in I}\alpha_{i}^{2^{n+p}}V^{n+p}\left(e_{i}\right)$.
Or, d'après le théorème \ref{thm:Existence _des_UP} on a $\mathbb{F}\subset\mathbb{F}_{2^{p}}$
donc $\alpha_{i}^{2^{p}}=\alpha_{i}$ d'où $\sum_{i\in I}\alpha_{i}^{2^{n+p}}V^{n+p}\left(e_{i}\right)=\sum_{i\in I}\alpha_{i}^{2^{n}}V^{n}\left(e_{i}\right)=V^{n}\left(x\right)$.
Enfin comme par hypothèse il existe $i\in I$ tel que $\left|\mathcal{O}\left(e_{i}\right)\right|=n+p$,
d'après la proposition \ref{prop:Caracteris_cycl} on a que $e_{i}$
est $\left(n,p\right)$-périodique et donc l'algèbre $A$ est
$\left(n,p\right)$-périodique.\smallskip{}

$\left(i\right)\Rightarrow\left(iii\right)$ Par restriction
du corps des scalaires à $\mathbb{F}_{2}$, l'opérateur $V$
est linéaire sur le $\mathbb{F}_{2}$-espace $A$ et on a $A=\ker\left(V^{n}\left(V^{p}-id\right)\right)$
donc $A=\ker\left(V^{n}\right)\oplus\ker\left(V^{p}-id\right)$. 

Si $n\neq0$, de la minimalité du couple $\left(p,n\right)$
on déduit qu'il existe $x,y\in A$ tels que $V^{n+p-1}\left(x\right)\neq V^{n-1}\left(x\right)$
et $V^{n+p-1}\left(y\right)\neq V^{n}\left(y\right)$ donc $x\notin\ker\left(V^{n-1}\right)$
et $y\notin\ker\left(V^{p-1}-id\right)$, dans le premier cas
on a $\left|\mathcal{O}\left(x\right)\right|=n$ et on complète
$\left\{ x\right\} $ en une base $\left(a_{i}\right)_{i\in I}$
de $\ker\left(V^{n}\right)$, dans le second cas on a $\left|\mathcal{O}\left(y\right)\right|=p$
et on complète $\left\{ y\right\} $ en une base $\left(b_{j}\right)_{j\in J}$
de $\ker\left(V^{p}-id\right)$. Et après extension du corps
des scalaires à $\mathbb{F},$ on a obtenu une base $\left(a_{i}\right)_{i\in I}\cup\left(b_{j}\right)_{j\in J}$
de $A$ remplissant les conditions de l'énoncé. 

Si $n=0$, on suit un raisonnement analogue pour $A=\ker\left(V^{p}-id\right)$.\smallskip{}

$\left(iii\right)\Rightarrow\left(i\right)$ Soit $x=\sum_{i\in I}\alpha_{i}a_{i}+\sum_{j\in J}\beta_{j}b_{j}$,
on a compte tenu des hypothèses
\[
V^{n+p}\left(x\right)=\sum_{i\in I}\alpha_{i}^{2^{n+p}}V^{n+p}\left(a_{i}\right)+\sum_{j\in J}\beta_{j}^{2^{n+p}}V^{n+p}\left(b_{j}\right)=\sum_{j\in J}\beta_{j}^{2^{n+p}}V^{n+p}\left(b_{j}\right)
\]
et
\[
V^{n}\left(x\right)=\sum_{i\in I}\alpha_{i}^{2^{n}}V^{n}\left(a_{i}\right)+\sum_{j\in J}\beta_{j}^{2^{n}}V^{n}\left(b_{j}\right)=\sum_{j\in J}\beta_{j}^{2^{n}}V^{n}\left(b_{j}\right).
\]
 Or d'après le théorème \ref{thm:Existence _des_UP} on a $\beta_{j}^{2^{n+p}}=\beta_{j}^{2^{n}}$
par conséquent $V^{n+p}\left(x\right)=V^{n}\left(x\right)$ pour
tout $x\in A$. 

Montrons que le couple $\left(p,n\right)$ est minimal pour l'ordre
lexicographique. Soient $\left(i,j\right)\in I\times J$ tel
que $\left|\mathcal{O}\left(a_{i}\right)\right|=n$ et $\left|\mathcal{O}\left(b_{j}\right)\right|=p$,
ceci implique que $V^{k}\left(a_{i}\right)\neq0$ quel que soit
$0\leq k<n$. On pose $z=a_{i}+b_{j}$, montrons que $V^{m+q}\left(z\right)\neq V^{m}\left(z\right)$
pour tout $\left(m,q\right)\in\mathbb{N}\times\mathbb{N}^{*}$
tel que $\left(q,m\right)\prec\left(p,n\right)$. Supposons qu'il
existe un couple $\left(m,q\right)$ tel que $\left(q,m\right)\prec\left(p,n\right)$
et $V^{m+q}\left(z\right)=V^{m}\left(z\right)\;\left(*\right)$. 

Si $q=p$, de $\left(q,m\right)\prec\left(p,n\right)$ il résulte
$m<n$, or $V^{m+p}\left(z\right)=V^{m+p}\left(a_{i}\right)+V^{m}\left(b_{j}\right)$
et $V^{m}\left(z\right)=V^{m}\left(a_{i}\right)+V^{m}\left(b_{j}\right)$
donc $V^{m+q}\left(a_{i}\right)=V^{m}\left(a_{i}\right)\;\left(**\right)$,
on a $m+q<n$ sinon on aurait $V^{m}\left(a_{i}\right)=V^{m+q}\left(a_{i}\right)=0$
avec $m<n$, alors en composant la relation $\left(**\right)$
par $V^{n-m-q}$ on obtient $V^{n-q}\left(a_{i}\right)=V^{n}\left(a_{i}\right)=0$
avec $n-q<n$, ce qui contredit l'hypothèse $\left|\mathcal{O}\left(a_{i}\right)\right|=n$. 

On a donc $q<p$. Si $m<n$ en composant $\left(*\right)$ par
$V^{kp-m}$ avec $kp>n$ on a $V^{kp+q}\left(z\right)=V^{kp}\left(z\right)$
ce qui entraîne $V^{q}\left(b_{j}\right)=b_{j}$ et donc $\left|\mathcal{O}\left(b_{j}\right)\right|<p$,
d'où une contradiction. Si $q<p$ et $m\geq n$ alors $V^{m+q}\left(z\right)=V^{m+q}\left(b_{j}\right)$
et $V^{m}\left(z\right)=V^{m}\left(b_{j}\right)$ on en déduit
que $\left|\mathcal{O}\left(V^{m}(b_{j}\right)\right|\leq q$,
or d'après la proposition \ref{prop:Caracteris_cycl}, l'élément
$b_{j}$ est $\left(0,p\right)$-périodique ce qui entraîne d'après
la proposition \ref{coroll:Egalit_des_O} que $\left|\mathcal{O}\left(V^{m}(b_{j}\right)\right|=\left|\mathcal{O}\left(b_{j}\right)\right|$,
d'où une contradiction.
\end{proof}
Il résulte aussitôt du résultat (\emph{iii}) ci-dessus que
\begin{cor}
Pour une algèbre $\left(n,p\right)$-périodique $A$ on a $\dim\left(A\right)\geq n+p.$
\end{cor}

\begin{lem}
Soient $m=2^{s}q$ avec $s\geq0$ et $q$ impair; $\sigma$ le
cycle $\left(1,\ldots,m\right)$ et $A$ un $\mathbb{F}$-espace
vectoriel $A$ de base $\left(e_{1},\ldots,e_{m}\right)$ où
$\mathbb{F}\subset\mathbb{F}_{2^{m}}$. 

Pour la structure d'algèbre définie sur $A$ par $e_{i}^{2}=e_{\sigma\left(i\right)}$
et les autres produits pris arbitrairement dans $A$, l'opérateur
d'évolution est $\left(0,m\right)$-périodique.
\end{lem}

\begin{proof}
Pour $x=\sum_{i=1}^{m}\alpha_{i}e_{i}$ on a $V^{m}\left(x\right)=\sum_{i=1}^{m}\alpha_{i}^{2^{m}}e_{\sigma^{m}\left(i\right)}=x$
et pour tout $k<m$ on a $V^{k}\left(e_{1}\right)=e_{k}\neq e_{1}$.
\end{proof}
\begin{prop}
Soient $n$ et $p$ deux entiers tels que $n\geq0$, $p=2^{r}q$
où $r\geq0$ et $q$ impair. Alors une $\mathbb{F}$-algèbre
$\left(n,p\right)$-périodique de dimension $d$ est semi-isomorphe
à une $\mathbb{F}$-algèbre, notée $A_{\left(\mathbf{s},\mathbf{t}\right)}$,
définie par la donnée:\smallskip{}

– d'un $M$-uplet d'entiers, $\mathbf{s}=\left(s_{1},\ldots,s_{M}\right)$
tel que $M\geq1$, $n=s_{1}\geq\ldots,\geq s\geq1$,\smallskip{}

– d'un $N$-uplet d'entiers, $\mathbf{t}=\left(t_{1},\ldots,t_{N}\right)$
tel que $N\geq1$, $r=t_{1}\geq\ldots\geq t_{N}\geq0$,\smallskip{}

tels que les composantes de $\mathbf{s}$ et de $\mathbf{t}$
vérifient
\[
\left(s_{1}+\cdots+s_{M}\right)+\left(2^{t_{1}}+\cdots+2^{t_{N}}\right)q=d;
\]

– d'une base $\bigcup_{i=1}^{M}\left\{ a_{i,j};1\leq j\leq s_{i}\right\} \cup\bigcup_{i=1}^{N}\left\{ b_{i,j};1\leq j\leq2^{t_{i}}q\right\} $,\smallskip{}

– de la table de multiplication:
\begin{align*}
a_{i,j}^{2} & =\begin{cases}
a_{i,j+1} & \text{si }1\leq j<s_{i}-1\\
0 & \text{si }j=s_{i}
\end{cases},\quad\left(i=1,\ldots,M\right),\\
b_{i,j}^{2} & =b_{i,\sigma_{i}\left(j\right)},\quad\left(i=1,\ldots,N;j=1,\ldots,2^{t_{i}}q\right),
\end{align*}

où $\sigma_{i}$ dénote le cycle $\left(1,\ldots,2^{t_{i}}q\right)$
et les produits non mentionnés sont pris arbitrairement dans
$A$.\medskip{}

Enfin, deux algèbres $A_{\left(\mathbf{s},\mathbf{t}\right)}$
et $A_{\left(\mathbf{s'},\mathbf{t}'\right)}$ sont semi-isomorphes
si et seulement si $\left(\mathbf{s},\mathbf{t}\right)=\left(\mathbf{s}',\mathbf{t}'\right)$.
\end{prop}

\begin{proof}
En restreignant le corps des scalaires à $\mathbb{F}_{2}$, sur
le $\mathbb{F}_{2}$-espace $A$ l'opérateur $V$ est linéaire
et on a $A=\ker V^{n}\left(V^{p}-id\right)=\ker V^{n}\oplus\ker\left(V^{q}-id\right)^{2^{r}}$.
La restriction de $V$ à $\ker V^{n}$ est nilpotent de degré
$n$, l'existence et les propriétés de la base $\bigcup_{i=1}^{M}\left\{ a_{i,j};1\leq j\leq s_{i}\right\} $
de $\ker V^{n}$ découlent de la proposition \ref{prop:Base_Alg_Res}.
Pour le sous-espace $\ker\left(V^{q}-id\right)^{2^{r}}$, la
décomposition de Frobenius appliquée à la restriction de $V$
à $\ker\left(V^{q}-id\right)^{2^{r}}$ fournit une décomposition
en somme directe de sous-espaces cycliques, par réunion de leurs
bases on obtient une base $\bigcup_{i=1}^{N}\left\{ b_{i,j};1\leq j\leq2^{t_{i}}q\right\} $
de $\ker\left(V^{q}-id\right)^{2^{r}}$ vérifiant les produits
donnés dans l'énoncé. D'après le lemme ci-dessus, les restrictions
de $V$ aux sous-espaces engendrés par $\left\{ b_{i,j};1\leq j\leq2^{t_{i}}q\right\} $
sont $\left(0,2^{t_{i}}q\right)$-périodiques, en particulier
pour $i=1$ elle est $\left(0,p\right)$-périodique. D'après
la proposition \ref{prop:base_UP}, on peut affirmer que l'algèbre
$A_{\left(\mathbf{s},\mathbf{t}\right)}$ est $\left(n,p\right)$-périodique.
Enfin la suite des invariants de similitude de ces deux restrictions
étant respectivement $\left(X^{s_{1}},\ldots,X^{s_{M}}\right)$
et $\bigl(\left(X^{q}-1\right)^{2^{t_{1}}},\ldots,\left(X^{q}-1\right)^{t_{N}}\bigr)$
on en déduit aussitôt que les espaces $A_{\left(\mathbf{s},\mathbf{t}\right)}$
et $A_{\left(\mathbf{s'},\mathbf{t}'\right)}$ sont isomorphes
si et seulement si $\mathbf{s}=\mathbf{s}'$ et $\mathbf{t}=\mathbf{t}'$.
Ces résultats étant conservés par extension du corps des scalaires
de $\mathbb{F}_{2}$ à $\mathbb{F}$, le résultat est prouvé. 
\end{proof}
\medskip{}

\subsection{Opérateurs d'évolution train pléniers}
\begin{defn}
Soit $A$ une $\mathbb{F}$-algèbre. L'algèbre $A$ est \emph{train
plénière de degré} $n$ s'il existe un polynôme $T\in\mathbb{F}\left[X\right]$
unitaire de degré $n\geq1$ tel que $T\left(V\right)\left(x\right)=0$
pour tout $x\in A$ et si pour tout $S\in\mathbb{F}\left[X\right]$
de degré $<n$ on a $S\left(V\right)\neq0$. Et dans ce cas on
dit que l'opérateur d'évolution $V:A\rightarrow A$ est train
plénier de degré $n$.
\end{defn}

Il résulte de la définition que le polynôme $T$ vérifiant $T\left(V\right)=0$
est unique, en effet s'il existe $R\in\mathbb{F}\left[X\right]$,
$R\neq T$ unitaire de degré $n$ tel que $R\left(V\right)=0$
alors on a $\left(T-R\right)\left(V\right)=0$ avec $T-R\neq0$
de degré $<n$, contradiction. Ce polynôme $T$ est appelé le
train polynôme de $A$

\medskip{}

\begin{rem}
Parmi les algèbres train plénières on trouve les algèbres nil-plénières
dont les train polynômes sont du type $T\left(X\right)=X^{n}$,
$n\geq2$ et les algèbres ultimement périodiques qui vérifient
des identités de la forme $X^{n+p}-X^{n}$ qui ne sont pas nécessairement
des train polynômes de ces algèbres. En effet les conditions
de minimalité de degré n'étant pas les mêmes une algèbre peut
être à la fois ultimement périodique et train plénière sans que
les polynômes soient identiques. Considérons par exemple la $\mathbb{F}_{2}$-algèbre
$A$ de base $\left(e_{1},e_{2},e_{3}\right)$ définie par $V\left(e_{1}\right)=e_{2}$,
$V\left(e_{2}\right)=e_{3}$ et $V\left(e_{3}\right)=e_{1}+e_{3}$,
pour $x=\alpha_{1}e_{1}+\alpha_{2}e_{2}+\alpha_{3}e_{3}$ on
trouve $V\left(x\right)=\alpha_{3}e_{1}+\alpha_{1}e_{2}+\left(\alpha_{2}+\alpha_{3}\right)e_{3}$,
$V^{2}\left(x\right)=\left(\alpha_{2}+\alpha_{3}\right)e_{1}+\alpha_{3}e_{2}+\left(\alpha_{1}+\alpha_{2}+\alpha_{3}\right)e_{3}$
et $V^{3}\left(x\right)=\left(\alpha_{1}+\alpha_{2}+\alpha_{3}\right)e_{1}+\left(\alpha_{2}+\alpha_{3}\right)e_{2}+\left(\alpha_{1}+\alpha_{2}\right)e_{3}$,
on a donc $V^{3}\left(x\right)+V^{2}\left(x\right)+x=0$ et $A$
est train plénière de degré 3. Mais par ailleurs dans $\mathbb{F}_{2}\left[X\right]$,
le polynôme $X^{3}+X^{2}+1$ divise $X^{8}-X$ on peut affirmer
que l'algèbre $A$ est $\left(1,7\right)$-périodique. 
\end{rem}

Compte tenu de cette remarque, les algèbres considérées dans
la suite, sont supposées non nil-plénières et les polynômes considérés
ont des degrés et des valuations distincts.
\begin{example}
1) En utilisant l'algèbre définie au 1) de l'exemple \ref{exa:Automates}
on montre que dans le cas $n=5$ on a $V^{5}=V^{3}+V$, pour
$n=7$ on montre que $V^{7}=V^{5}+V$ et quand $n=13$ on a $V^{13}=V^{11}+V^{9}+V^{3}+V$.

2) Pour l'algèbre définie au 2) de l'exemple \ref{exa:Automates}
on montre que pour $n=5$ l'opérateur d'évolution $V$ vérifie
$T\left(V\right)=0$ pour $T\left(X\right)=X^{5}-X^{4}-X^{3}-X^{2}-X-1$,
pour $n=7$ en prenant $T\left(X\right)=X^{7}-X^{6}-X^{3}-X^{2}-X-1$
on a $T\left(V\right)=0$.
\end{example}

On a dans les algèbres train plénières un résultat analogue à
celui de la proposition \ref{prop:Dim_finie=000026UP} pour les
algèbres ultimement périodiques.
\begin{prop}
\label{prop:dim_finie=000026TP}Si le corps $\mathbb{F}$ est
fini alors toute $\mathbb{F}$-algèbre de dimension finie est
train plénière.
\end{prop}

\begin{proof}
Il suffit de reprendre le début de la preuve de la proposition
\ref{prop:Dim_finie=000026UP}, le polynôme $T$ construit à
partir de l'algèbre $A$ est le train polynôme de $A$.
\end{proof}
\begin{lem}
Soit $A$ une algèbre train plénière de degré $n$ de train polynôme
$T$, il existe $\widehat{x}\in A$ vérifiant $T\left(V\right)\left(\widehat{x}\right)=0$
et $S\left(V\right)\left(\widehat{x}\right)\neq0$ pour tout
$S\in\mathbb{F}\left[X\right]$ de degré $<n$.
\end{lem}

\begin{proof}
Soit $T=\prod_{k=1}^{r}T_{k}^{\nu_{k}}$ la décomposition de
$T$ en facteurs irréductibles dans $\mathbb{F}\left[X\right]$,
en utilisant le fait que $V$ est un morphisme du groupe additif
on a $A=\ker T\left(V\right)=\bigoplus_{k=1}^{r}\ker T_{k}^{\nu_{k}}\left(V\right)$.
Montrons que pour chaque $1\leq k\leq r$, il existe un élément
$x_{k}\in\ker T_{k}^{\nu_{k}}\left(V\right)\setminus\ker T_{k}^{\nu_{k}-1}\left(V\right)$.
En effet, si pour un entier $k$ on a $T_{k}^{\nu_{k}-1}\left(V\right)\left(x_{k}\right)=0$
pour tout $x_{k}\in\ker T_{k}^{\nu_{k}}\left(V\right)$, alors
en considérant le polynôme $R_{k}=T_{k}^{\nu_{k}-1}\prod_{i\neq k}T_{i}^{\nu_{i}}$,
comme pour tout $x\in A$ on a $x=\sum_{i=1}^{r}x_{i}$ où $x_{i}\in\ker T_{i}^{\nu_{i}}\left(V\right)$
on obtient $R_{k}\left(V\right)\left(x\right)=0$ avec $R_{k}$
de degré $<n$, ce qui contredit la définition de $T$. 

Posons $\widehat{x}=\sum_{k=1}^{r}x_{k}$ où $x_{k}\in\ker T_{k}^{\nu_{k}}\left(V\right)\setminus\ker T_{k}^{\nu_{k}-1}\left(V\right)$
pour $1\leq k\leq r$. Alors pour tout entier $1\leq k\leq r$
on a $R_{k}\left(V\right)\left(\widehat{x}\right)\neq0$, car
si on a $R_{k}\left(V\right)\left(\widehat{x}\right)=0$ pour
un entier $k$, alors on a $\widehat{x}\in\ker T_{k}^{\nu_{k}-1}\left(V\right)\oplus\ker\left(\prod_{i\neq k}T_{i}^{\nu_{i}}\right)\left(V\right)$
ce qui implique $\widehat{x}=0$, d'où une contradiction. On
en déduit que pour tout polynôme $D$ diviseur de $T$ de degré
$<n$, on a $D\left(V\right)\left(\widehat{x}\right)\neq0$,
en effet s'il existe un diviseur $D$ de $T$ de degré $<n$
tel que $D\left(V\right)\left(\widehat{x}\right)=0$ comme $D$
divise l'un des polynômes $R_{k}$ on aurait $R_{k}\left(V\right)\left(\widehat{x}\right)=0$.
Enfin si on suppose qu'il existe $S\in\mathbb{F}\left[X\right]$
de degré $<n$ tel que $S\left(V\right)\left(\widehat{x}\right)=0$,
soit $D=\text{gcd}\left(T,S\right)$, si $D\neq1$ alors $D$
serait un diviseur de $T$ de degré $<n$ tel que $D\left(V\right)\left(\widehat{x}\right)=0$
et si $D=1$ alors on aurait $\widehat{x}=0$, dans les deux
cas on a une contradiction, par conséquent on a montré que pour
tout $S\in\mathbb{F}\left[X\right]$ de degré $<n$ on a $S\left(V\right)\left(\widehat{x}\right)\neq0$.
\end{proof}
Pour qu'une $\mathbb{F}$-algèbre soit train plénière il faut
une condition sur le corps $\mathbb{F}$. \medskip{}

\begin{defn}
Etant donné un polynôme $T\in K\left[X\right]$, $T\left(X\right)=\sum_{k=0}^{n}\alpha_{k}X^{k}$
qui n'est pas un monôme. Soit $E\left(T\right)=\left\{ k;\alpha_{k}\neq0\right\} $
l'ensemble des exposants de $T$, on appelle \emph{striction}
du polynôme $T$, le nombre $\sigma\left(T\right)$ défini par
$\sigma\left(T\right)=\gcd\left\{ j-i;i,j\in E\left(T\right),i<j\right\} $.
\end{defn}

\begin{thm}
\label{thm:Existence_TP}Soit $T\in\mathbb{F}\left[X\right]$
un polynôme de striction $\sigma\left(T\right)$. Si $A$ est
une $\mathbb{F}$-algèbre train plénière de train polynôme $T$
alors on a $\mathbb{F}\subset\mathbb{F}_{2^{\sigma\left(T\right)}}$.
\end{thm}

\begin{proof}
Soit $T\left(X\right)=X^{n}-\sum_{k=0}^{n-1}\alpha_{k}X^{k}$
avec $\alpha_{k}\in\mathbb{F}$. D'après le lemme ci-dessus,
il existe $\widehat{x}\in A$ tel que $T\left(V\right)\left(\widehat{x}\right)=0$
et $S\left(V\right)\left(\widehat{x}\right)\neq0$ pour polynôme
$S$ de degré $<n$. De $T\left(V\right)\left(\widehat{x}\right)=0$
il vient $V^{n}\left(\widehat{x}\right)-\sum_{k=0}^{n-1}\alpha_{k}V^{k}\left(\widehat{x}\right)=0$
et pour tout $\lambda\in\mathbb{F}$ de $T\left(V\right)\left(\lambda\widehat{x}\right)=0$
il résulte $\lambda^{2^{n}}V^{n}\left(\widehat{x}\right)-\sum_{k=0}^{n-1}\lambda^{2^{k}}\alpha_{k}V^{k}\left(\widehat{x}\right)=0$,
on en déduit que $\sum_{k=0}^{n-1}\alpha_{k}\bigl(\lambda^{2^{n}}-\lambda^{2^{k}}\bigr)V^{k}\left(\widehat{x}\right)=0$,
or $S\left(X\right)=\sum_{k=0}^{n-1}\alpha_{k}\bigl(\lambda^{2^{n}}-\lambda^{2^{k}}\bigr)X^{k}$
est un élément de $\mathbb{F}\left[X\right]$ de degré $<n$
tel que $S\left(V\right)\left(\widehat{x}\right)=0$ par conséquent
on a $S=0$. Il en résulte que $\alpha_{k}\bigl(\lambda^{2^{n}}-\lambda^{2^{k}}\bigr)=0$
pour tout $0\leq k<n$ et donc $\lambda^{2^{n}}-\lambda^{2^{k}}=0$
pour tout $k\in E\left(T\right)$, on en déduit que $\lambda^{2^{i}}-\lambda^{2^{j}}=0$
pour tout $i,j\in E\left(T\right)$ ce qui entraîne $\lambda^{2^{j-i}}-\lambda=0$
pour tout $i,j\in E\left(T\right)$ tels que $i<j$, par conséquent
le scalaire $\lambda$ est racine des polynômes $X^{2^{j-i}}-X$,
autrement dit $F\subset F_{2^{j-i}}$ pour tout $i,j\in E\left(T\right)$
tels que $i<j$ ce qui équivaut à $\mathbb{F}\subset\mathbb{F}_{2^{\sigma\left(T\right)}}$. 
\end{proof}
On en déduit une description des train polynômes.
\begin{cor}
Soit $A$ une $\mathbb{F}_{2^{m}}$-algèbre, un polynôme $T\in\mathbb{F}_{2^{m}}\left[X\right]$
de degré $n$ et de valuation $v$ est un train polynôme de $A$,
s'il existe un entier $\sigma\geq1$ tel que $m$ divise $\sigma$,
$\sigma$ divise $n-v$ et 
\[
T\left(X\right)=\sum_{k=0}^{\nicefrac{\left(n-v\right)}{\sigma}}\alpha_{v+k\sigma}X^{v+k\sigma},
\]
avec $E\left(T\right)\subset\left\{ v+k\sigma;0\leq k\leq\nicefrac{\left(n-v\right)}{\sigma}\right\} $,
$\alpha_{v}\neq0$ et $\alpha_{n}=1$.
\end{cor}

\begin{proof}
Soit $T\in\mathbb{F}_{2^{m}}\left[X\right]$, de degré $n$,
de valuation $v$ et de striction $\sigma$ un train polynôme
de $A$. D'après le théorème \ref{thm:Existence_TP} pour que
$T$ soit un train polynôme de $A$ il faut que $\mathbb{F}_{2^{m}}\subset\mathbb{F}_{2^{\sigma}}$
ce qui implique que $m$ divise $\sigma$. Soient $n_{0},\ldots,n_{p}$
les éléments de $E\left(T\right)$ indexés de façon que $n_{0}<\cdots<n_{p}$,
alors $n_{0}$ est la valuation de $T$ et $n_{p}$ son degré.
On a $T\left(X\right)=\sum_{i=0}^{p}\alpha_{n_{i}}X^{n_{i}}$,
par définition de la striction de $T$, pour tout $0\leq i\leq p$
il existe un entier $k_{i}$ tel que $n_{i}=n_{0}+k_{i}\sigma$,
il en résulte que $n_{i+1}-n_{i}=\left(k_{i+1}-k_{i}\right)\sigma$
d'où $\sum_{i=0}^{p-1}\left(n_{i+1}-n_{i}\right)=\sigma\sum_{i=0}^{p-1}\left(k_{i+1}-k_{i}\right)$,
comme $n_{0}=v$, $n_{p}=n$ et $k_{0}=0$, on a $n-v=k_{p}\sigma$
d'où l'on conclut que $\sigma$ divise $n-v$ et que $E\left(T\right)\subset\left\{ v+k\sigma;0\leq k\leq\nicefrac{\left(n-v\right)}{\sigma}\right\} $.
Enfin en prenant $\alpha_{v+k\sigma}=0$ quand $v+kv\notin E\left(T\right)$
on obtient l'expression de $T$ donnée dans l'énoncé.
\end{proof}
On a aussi un résultat de factorisation des train polynômes.
\begin{prop}
Les train polynômes des $\mathbb{F}$-algèbres train plénières
se décomposent en produit de polynômes irréductibles dans $\mathbb{F}_{2}\left[X\right]$. 
\end{prop}

\begin{proof}
Soit $T\in\mathbb{F}\left[X\right]$ le train polynôme d'une
$\mathbb{F}$-algèbre train plénière de degré $n$, montrons
qu'il existe des entiers $s\geq n$ et $m\geq1$ tels que $T$
divise le polynôme $X^{s+m}-X^{s}$. Pour tout entier $k\geq n$
on note $R_{k}$ le reste de la division euclidienne dans $\mathbb{F}\left[X\right]$
de $X^{k}-X$ par $T$. Il résulte du théorème \ref{thm:Existence_TP}
que le corps $\mathbb{F}$ est fini, par conséquent l'ensemble
$\left\{ R_{k};k\geq n\right\} $ est fini de cardinal au plus
$\left|\mathbb{F}\right|^{n}$, il existe donc deux entiers $s$
et $t$ tels que $n\leq s<t$ et $R_{s}=R_{t}$ alors $T$ divise
le polynôme $\left(X^{t}-X\right)-\left(X^{s}-X\right)=X^{t}-X^{s}$
que l'on peut écrire sous la forme $X^{s+m}-X^{s}$ en posant
$m=t-s$. 

Soit $m=2^{a}b$ avec $b$ un entier impair, on a $X^{s}\left(X^{m}-1\right)=X^{s}\left(X^{b}-1\right)^{2^{a}}$.
Or on a la décomposition $X^{b}-1=\prod_{d\left|b\right.}\Phi_{d}\left(X\right)$
où $\Phi_{d}\in\mathbb{F}_{2}\left[X\right]$ est le $d$-ième
polynôme cyclotomique, de plus $\Phi_{d}$ admet une factorisation
en polynômes unitaires irréductibles dans $\mathbb{F}_{2}\left[X\right]$. 
\end{proof}
Pour la classification des train algèbres plénières on utilisera
les définitions suivantes.
\begin{defn}
Soit $A$ une $\mathbb{F}_{2^{p}}$-train algèbre plénière vérifiant
le train polynôme $T=X^{v}f_{1}^{r_{1}}\ldots f_{m}^{r_{m}}$
où $f_{1},\ldots,f_{m}$ sont des polynômes unitaires irréductibles
dans $\mathbb{F}_{2}\left[X\right]$ tels que $1\leq\deg f_{k}<\deg f_{k+1}$
et $f_{1}\left(0\right)=\cdots=f_{m}\left(0\right)=1$. \smallskip{}

– On dit qu'une partition $\left\{ I_{1},\ldots,I_{s}\right\} $
de $\left\llbracket 1,m\right\rrbracket $ est $\mathbb{F}_{2^{p}}$-\emph{compatible}
si $p$ divise $\sigma\left(\prod_{i\in I_{k}}f_{i}^{r_{i}}\right)$
pour tout $1\leq k\leq s$ et si la partition $\left\{ I_{1},\ldots,I_{s}\right\} $
vérifiant cette condition est la plus fine.\smallskip{}

– Etant donnée une partition $\left\{ I_{1},\ldots,I_{s}\right\} $
de $\left\llbracket 1,m\right\rrbracket $ qui est $\mathbb{F}_{2^{p}}$-compatible,
soit $\mathcal{S}$ l'application qui à $I_{k}$ associe l'ensemble
$\mathcal{D}_{k}$ des diviseurs de degré $\text{\ensuremath{\geq}1}$
du polynôme $\prod_{i\in I_{k}}f_{i}^{r_{i}}$ tel que $\mathcal{D}_{k}$
soit totalement ordonné pour la relation de divisibilité dans
$\mathbb{F}_{2}\left[X\right]$ et $p$ divise $\sigma\left(D\right)$
pour tout $D\in\mathcal{D}_{k}$. L'ensemble $\left\{ \mathcal{S}\left(I_{1}\right),\ldots,\mathcal{S}\left(I_{s}\right)\right\} $
est appelé une \emph{subdivision} $\mathbb{F}_{2^{p}}$-\emph{compatible}
associée à la partition $\left\{ I_{1},\ldots,I_{s}\right\} $. 
\end{defn}

\begin{example}
1) Soit $T=X^{12}+X^{10}+X^{4}+1$, on a $\sigma\left(T\right)=2$
par conséquent une $\mathbb{F}_{2^{p}}$-algèbre vérifie $T$
si $p$ divise $2$. On a $T=\left(X^{6}+X^{5}+X^{2}+1\right)^{2}=f_{1}{}^{2}f_{2}{}^{2}f_{3}{}^{2}$
où $f_{1}=X+1$, $f_{2}=X^{2}+X+1$ et $f_{3}=X^{3}+X^{2}+1$,
on a $\sigma\left(f_{i}^{2}\right)=2$, $i=1,2,3$ et la partition
$\left\{ \left\{ 1\right\} ,\left\{ 2\right\} ,\left\{ 3\right\} \right\} $
est $\mathbb{F}_{2^{p}}$-compatible pour $p=1$ ou $p=2$. Dans
le cas $p=1$, la subdivision $\mathbb{F}_{2}$-compatible de
cette partition est $\left\{ \left\{ f_{1},f_{1}^{2}\right\} ,\left\{ f_{2},f_{2}^{2}\right\} ,\left\{ f_{3},f_{3}^{2}\right\} \right\} $,
pour $p=2$ la subdivision $\mathbb{F}_{4}$-compatible de cette
partition est $\left\{ \left\{ f_{1}^{2}\right\} ,\left\{ f_{2}^{2}\right\} ,\left\{ f_{3}^{2}\right\} \right\} $.\smallskip{}

2) Soit $T=X^{6}+1$, on a $\sigma\left(T\right)=6$ donc une
$\mathbb{F}_{2^{p}}$-algèbre vérifie $T$ si $p$ divise $6$.
Or $T=\left(X^{3}+1\right)^{2}=f_{1}\left(X\right)^{2}f_{2}\left(X\right)^{2}$
avec $f_{1}=X+1$ et $f_{2}=X^{2}+X+1$, donc $\sigma\left(f_{i}^{2}\right)=2$.
Par conséquent si $p=1$ ou $p=2$ la partition $\left\{ \left\{ 1\right\} ,\left\{ 2\right\} \right\} $
est $\mathbb{F}_{2^{p}}$-compatible, dans le cas $p=1$ une
subdivision $\mathbb{F}_{2}$-compatible de cette partition est
$\left\{ \left\{ f_{1},f_{1}^{2}\right\} ,\left\{ f_{2},f_{2}^{2}\right\} \right\} $,
tandis que dans le cas $p=2$ la subdivision $\mathbb{F}_{4}$-compatible
de cette partition est $\left\{ \left\{ f_{1}^{2}\right\} ,\left\{ f_{2}^{2}\right\} \right\} $.
Si $p=3$ c'est la partition $\left\{ \left\{ 1,2\right\} \right\} $
qui est $\mathbb{F}_{8}$-compatible et qui a pour subdivision
$\mathbb{F}_{8}$-compatible $\bigl\{ f_{1}f_{2},\left(f_{1}f_{2}\right)^{2}\bigr\}$.\smallskip{}

3) Soit $T=X^{12}+X^{9}+X^{3}+1$, on a $\sigma\left(T\right)=3$
par conséquent une $\mathbb{F}_{2^{p}}$-algèbre vérifie $T$
si $p=1$ ou $p=3$ . On a $T=f_{1}^{2}f_{2}^{2}f_{3}$ où $f_{1}=X+1$,
$f_{2}=X^{2}+X+1$ et $f_{3}=X^{6}+X^{3}+1$. Si $p=1$, la partition
$\left\{ \left\{ 1\right\} ,\left\{ 2\right\} ,\left\{ 3\right\} \right\} $
est $\mathbb{F}_{2}$-compatible et elle admet la subdivision
$\mathbb{F}_{2}$-compatible $\left\{ \left\{ f_{1},f_{1}^{2}\right\} ,\left\{ f_{2},f_{2}^{2}\right\} ,\left\{ f_{3}\right\} \right\} $.
Dans le cas $p=3$, on a $f_{1}f_{2}=X^{3}+1$, $f_{1}^{2}f_{2}=X^{4}+X^{3}+X+1$,
$f_{1}f_{2}^{2}=X^{5}+X^{4}+X^{3}+X^{2}+X+1$ et $f_{1}^{2}f_{2}^{2}=X^{6}+1$,
par conséquent la partition $\left\{ \left\{ 1,2\right\} ,\left\{ 3\right\} \right\} $
est $\mathbb{F}_{8}$-compatible et la subdivision $\mathbb{F}_{8}$-compatible
associée est $\bigl\{\bigl\{ f_{1}f_{2},\left(f_{1}f_{2}\right)^{2}\bigr\},\left\{ f_{3}\right\} \bigr\}$.
\end{example}

Les résultats suivants seront utilisés pour la classification
des train algèbres plénières.
\begin{lem}
\label{lem:C_D}Etant donnés 

$\bullet$ $P\in\mathbb{F}_{2}\left[X\right]$ , $P\left(X\right)=X^{n}+\sum_{k=v}^{n-1}\alpha_{k}X^{k}$
un polynôme unitaire de valuation $v\ne n$ et de striction $\sigma\left(P\right)$; 

$\bullet$ $A$ un $\mathbb{F}$-espace vectoriel et $\mathscr{B}=\left(e_{0},\ldots,e_{n-1}\right)$
une base de $A$.

$\bullet$ La structure de $\mathbb{F}$-algèbre définie sur
$A$ par la donnée de
\[
e_{i}^{2}=C_{P}\left(e_{i}\right),\quad\left(i=0,\ldots,n-1\right)
\]

où $C_{P}$ est l'endomorphisme défini sur $A$ par:
\[
C_{P}\left(e_{i}\right)=\begin{cases}
e_{i+1} & \text{si }0\leq i<n-1,\medskip\\
\sum_{k=v}^{^{n-1}}\alpha_{k}e_{k} & \text{si }i=n-1.
\end{cases}
\]

les produits non mentionnés étant pris arbitrairement dans $A$.\smallskip{}

Alors on a $P\left(V\right)\left(x\right)=0$ pour tout $x\in A$
si et seulement si $\mathbb{F}\subset\mathbb{F}_{2^{\sigma\left(P\right)}}$. 
\end{lem}

\begin{proof}
Pour tout $1\leq k\leq n-1$ on a $V^{k}\left(e_{0}\right)=e_{k}$
et $V^{n}\left(e_{0}\right)=V\left(e_{n-1}\right)=\sum_{k=v}^{n-1}\alpha_{k}e_{k}$
donc 
\[
V^{n}\left(e_{0}\right)=\sum_{k=v}^{n-1}\alpha_{k}V^{k}\left(e_{0}\right).\quad\left(\star\right)
\]
Il en résulte que $P\left(V\right)\left(e_{0}\right)=V^{n}\left(e_{0}\right)+\sum_{k=v}^{n-1}\alpha_{k}V^{k}\left(e_{0}\right)=0.$
Ensuite en appliquant l'opérateur $V$ à la relation $\left(\star\right)$
on obtient $V^{n}\left(e_{1}\right)=V^{n+1}\left(e_{0}\right)=\sum_{k=v}^{n-1}\alpha_{k}^{2}V^{k+1}\left(e_{0}\right)=\sum_{k=v}^{n-1}\alpha_{k}V^{k}\left(e_{1}\right)$
autrement dit $P\left(V\right)\left(e_{1}\right)=0$. En poursuivant
ainsi on montre que $P\left(V\right)\left(e_{i}\right)=0$ pour
tout $0\leq i\leq n-1$. 

Pour $x\in A$, $x=\sum_{i=0}^{n-1}\beta_{i}e_{i}$ avec $\beta_{i}\in\mathbb{F}$,
on trouve:
\begin{equation}
P\left(V\right)\left(x\right)=\sum_{i=0}^{n-1}\Bigl(\beta_{i}^{2^{n}}V^{n}\left(e_{i}\right)+\sum_{k=v}^{n-1}\alpha_{k}\beta_{i}^{2^{k}}V^{k}\left(e_{i}\right)\Bigr).\label{eq:P(V)(x)=00003D}
\end{equation}
Supposons $\mathbb{F}\subset\mathbb{F}_{2^{\sigma\left(P\right)}}$
et montrons que $P\left(V\right)\left(x\right)=0$ pour tout
$x\in A$. Pour tout $k\in E\left(P\right)$, l'entier $\sigma\left(P\right)$
divise $k-v$, alors avec $\mathbb{F}\subset\mathbb{F}_{2^{\sigma\left(P\right)}}$
on en déduit que $\beta_{i}^{2^{k}}=\beta_{i}^{2^{v}}$, par
conséquent pour tout $v\leq k<n-1$ on a $\alpha_{k}\beta_{i}^{2^{k}}=\alpha_{k}\beta_{i}^{2^{v}}$
d'où
\[
\beta_{i}^{2^{n}}V^{n}\left(e_{i}\right)+\sum_{k=v}^{n-1}\alpha_{k}\beta_{i}^{2^{k}}V^{k}\left(e_{i}\right)=\beta_{i}^{2^{v}}\Bigl(V^{n}\left(e_{i}\right)+\sum_{k=v}^{n-1}\alpha_{k}V^{k}\left(e_{i}\right)\Bigr)=\beta_{i}^{2^{v}}P\left(V\right)\left(e_{i}\right)=0
\]
 pour tout $0\leq i<n$, soit finalement $P\left(V\right)\left(x\right)=0$.

Réciproquement, si $P\left(V\right)\left(x\right)=0$ pour tout
$x\in A$, en prenant $x=\beta_{0}e_{0}$ dans la relation (\ref{eq:P(V)(x)=00003D})
on obtient $\beta_{0}^{2^{n}}V^{n}\left(e_{0}\right)+\sum_{k=v}^{n-1}\alpha_{k}\beta_{0}^{2^{k}}V^{k}\left(e_{0}\right)=0$
ou encore en utilisant la relation $\left(\star\right)$ on obtient
$\sum_{k=v}^{n-1}\alpha_{k}\left(\beta_{0}^{2^{n}}+\beta_{0}^{2^{k}}\right)e_{k}=0$
d'où il résulte que $\beta_{0}^{2^{n}}+\beta_{0}^{2^{k}}=0$
pour tout $k\in E\left(P\right)$ et tout $\beta_{0}\in\mathbb{F}$
ce qui implique $\beta_{0}^{2^{j-i}}=\beta_{0}$ pour tout $i,j\in E\left(P\right)$,
$i<j$ et $\beta_{0}\in\mathbb{F}$ et donc $\mathbb{F}\subset\mathbb{F}_{2^{\sigma\left(P\right)}}$. 
\end{proof}
\begin{prop}
\label{prop:C_Di=000026C_D}Soient $D\in\mathbb{F}_{2}\left[X\right]$
un polynôme unitaire qui n'est pas réduit à un monôme et $D_{1},\ldots,D_{m}\in\mathbb{F}_{2}\left[X\right]$
des diviseurs de $D$ tels que $D_{m}=D$ et $D_{k}$ divise
$D_{k+1}$ pour tout $1\leq k<m$. Soit $p\geq1$ un entier tel
que $p$ divise $\sigma\left(D_{k}\right)$ pour tout $1\leq k\leq m$. 

Alors le $\mathbb{F}$-espace vectoriel $A$ avec $\mathbb{F}\subset\mathbb{F}_{2^{p}}$,
de base $\bigcup_{i=0}^{m}\left\{ e_{i,j};0\leq j\leq d_{i}-1\right\} $
où $d_{i}=\deg D_{i}$, muni de la loi d'algèbre: 
\[
e_{i,j}^{2}=C_{D_{i}}\left(e_{i,j}\right)\quad(1\leq i\leq m,0\leq j\leq d_{i}-1),
\]
les produits non mentionnés étant arbitraires dans $A$, vérifie
$D\left(V\right)\left(x\right)=0$ pour tout $x\in A$.
\end{prop}

\begin{proof}
Notons $A_{i}$ le sous-espace engendré par $\left\{ e_{i,j};0\leq j\leq d_{i}-1\right\} $,
d'après le lemme \ref{lem:C_D} on a $D_{i}\left(V\right)\left(x_{i}\right)=0$
pour tout $x_{i}\in A_{i}$. Soit $x\in A$, on peut écrire $x=\sum_{i=1}^{m}x_{i}$
où $x_{i}\in A_{i}$ et on a $D\left(V\right)\left(x\right)=\sum_{i=1}^{m}D\left(V\right)\left(x_{i}\right)$,
or pour chaque $1\leq i\leq m$ il existe $Q_{i}\in\mathbb{F}_{2}\left[X\right]$
tel que $D=Q_{i}D_{i}$, il en résulte que $D\left(V\right)\left(x\right)=\sum_{i=1}^{m}Q_{i}\left(V\right)D_{i}\left(V\right)\left(x_{i}\right)=0$
pour tout $x\in A$.
\end{proof}
On peut maintenant donner une classification des train algèbres
plénières.
\begin{prop}
Soient $T\in\mathbb{F}_{2}\left[X\right]$ un polynôme unitaire
de degré $n$ de valuation $v\neq n$ et $p\geq1$ un entier
divisant $\sigma\left(T\right)$. Soit $\left\{ I_{1},\ldots,I_{s}\right\} $
une partition $\mathbb{F}_{2^{p}}$-compatible et $\left\{ \mathcal{S}\left(I_{1}\right),\ldots,\mathcal{S}\left(I_{s}\right)\right\} $
une subdivision $\mathbb{F}_{2^{p}}$-compatible associée à cette
partition. \smallskip{}

Si pour chaque $k=1,\ldots,s$ on note:\medskip{}

\begin{tabular}{ll}
\multirow{2}{*}{$\mathcal{S}\left(I_{k}\right)=\left\{ T_{\left(k,1\right)},\ldots,T_{\left(k,n_{k}\right)}\right\} $,} & où $1\leq n_{k}$, $T_{\left(k,i+1\right)}\mid T_{\left(k,i\right)}$,
$1\leq\deg T_{\left(k,n_{k}\right)}$ $\medskip$\tabularnewline
 & et $p\mid\sigma\left(T_{\left(k,i\right)}\right)$,$\medskip$\tabularnewline
$\delta_{\left(k,i\right)}=\deg T_{\left(k,i\right)}$, & pour $1\leq i\leq n_{k}$,$\medskip$\tabularnewline
$\mathbf{q}_{k}=\left(q_{\left(k,1\right)},\ldots,q_{\left(k,r_{k}\right)}\right)$, & où $1\leq r_{k}\leq n_{k}$ et  $1\leq q_{\left(k,1\right)},\ldots,q_{\left(k,r_{k}\right)}$.$\medskip$\tabularnewline
\end{tabular}

\medskip{}

Si on pose $\mathbf{v}=\left(v_{1},\ldots,v_{m}\right)$ où $v=v_{1}\geq\ldots\geq v_{m}\geq1$
et $\mathbf{q}=\left(\mathbf{q}_{1},\ldots,\mathbf{q}_{s}\right)$.\medskip{}

Alors une $\mathbb{F}_{2^{p}}$-algèbre de dimension $d$, train
plénière de train polynôme $T$ est semi-isomorphe à la $\mathbb{F}_{2^{p}}$-algèbre
$A\left(\mathbf{v},\mathbf{q}\right)$ définie par la donnée: 

– d'une base 
\[
\mathscr{B}=\bigcup_{i=1}^{m}\left\{ a_{i,j}/1\leq j\leq v_{i}\right\} \cup\bigcup_{k=1}^{s}\bigcup_{i=1}^{r_{k}}\bigcup_{j=1}^{q_{\left(k,i\right)}}\left\{ b_{\left(k,i\right),j,p}/0\leq p\leq\delta_{\left(k,i\right)}-1\right\} 
\]

telle que
\[
\sum_{i=1}^{m}v_{i}+\sum_{k=1}^{s}\sum_{i=1}^{r_{k}}\delta_{\left(k,i\right)}q_{\left(k,i\right)}=d,
\]

– et de la table de multiplication
\begin{align*}
a_{i,j}^{2} & =\begin{cases}
a_{i,j+1} & \text{si }1\leq j\leq v_{i}-1,\\
\;0 & \text{si }j=v_{i},
\end{cases},\quad i=1,\ldots,m,\\
\bigl(b_{\left(k,i\right),j,p}\bigr)^{2} & =C_{T_{\left(k,j\right)}}\bigl(b_{\left(k,i\right),j,p}\bigr),\qquad\begin{array}{c}
1\leq k\leq s;\;1\leq i\leq r_{k};\\
1\leq j\leq q_{\left(k,i\right)};\\
1\leq p\leq\delta_{\left(k,j\right)}-1.
\end{array}
\end{align*}
\end{prop}

\begin{proof}
Soit $A$ une $\mathbb{F}_{2^{p}}$-algèbre train plénière de
train polynôme $T$. Pour tout $1\leq k\leq s$ et tout $1\leq i\leq n_{k}-1$
on a $\deg T_{\left(k,i+1\right)}\leq\deg T_{\left(k,i\right)}$.
On a $T=X^{v}\prod_{k=1}^{s}T_{\left(k,1\right)}$, par définition
d'une subdivision $\mathbb{F}_{2^{p}}$-compatible les polynômes
$T_{\left(1,1\right)},\ldots,T_{\left(s,1\right)}$ sont premiers
entre eux, alors par restriction du corps des scalaires à $\mathbb{F}_{2}$
on munit $A$ d'une structure de $\mathbb{F}_{2}$-espace et
$A=\ker V^{v}\oplus\bigoplus_{k=1}^{s}\ker T_{\left(k,1\right)}\left(V\right)$.\medskip{}

La restriction de $V$ à $\ker V^{v}$ est nilpotente de degré
$v$, l'existence et les propriétés de la base $\bigcup_{i=1}^{m}\left\{ a_{i,j};1\leq j\leq v_{i}\right\} $
de $\ker V^{v}$ découlent de la proposition \ref{prop:base_UP}.
\medskip{}

Construisons une base de $\bigoplus_{k=1}^{s}\ker T_{\left(k,1\right)}\left(V\right)$
qui vérifie la table de multiplication donnée dans l'énoncé.
Pour tout $1\leq k\leq s$ on a $\ker T_{\left(k,1\right)}\left(V\right)\neq\left\{ 0\right\} $,
sinon $A$ vérifierait un train polynôme de degré $<n$. Considérons
la restriction de $V$ à $\ker T_{\left(k,1\right)}\left(V\right)$,
comme pour tout $1\leq i\leq n_{k}-1$ le polynôme $T_{\left(k,i+1\right)}$
divise $T_{\left(k,i\right)}$ on a donc $\ker T_{\left(k,i+1\right)}\left(V\right)\subset\ker T_{\left(k,i\right)}\left(V\right)$,
soit $E_{\left(k,i\right)}$ un supplémentaire de $\ker T_{\left(k,i+1\right)}\left(V\right)$
dans $\ker T_{\left(k,i\right)}\left(V\right)$, en notant $E_{\left(k,n_{k}\right)}=\ker T_{\left(k,n_{k}\right)}\left(V\right)$
on a la décomposition $\ker T_{\left(k,1\right)}\left(V\right)=\bigoplus_{i=1}^{n_{k}}E_{\left(k,i\right)}$.
Pour chaque entier $1\leq k\leq s$, on note $r_{k}\geq1$ le
plus grand entier tel que $E_{\left(k,r_{k}\right)}\neq\left\{ 0\right\} $,
pour chaque $1\leq i\leq r_{k}$ il existe un élément $b_{\left(k,i\right),1,0}\in E_{\left(k,i\right)}$,
$b_{\left(k,i\right),1,0}\neq0$, le polynôme $\prod_{k=1}^{s}T_{\left(k,1\right)}$
étant de valuation $0$ il en résulte que le polynôme $T_{\left(k,i\right)}$
est aussi de valuation $0$ et donc le système $\left\{ V^{j}\left(b_{\left(k,i\right),1,0}\right);0\leq j<\delta_{\left(k,i\right)}\right\} $
est libre. Si on note $b_{\left(k,i\right),1,j}=V^{j}\left(b_{\left(k,i\right),1,0}\right)$
on a $\left(b_{\left(k,i\right),1,j}\right)^{2}=V^{j+1}\left(b_{\left(k,i\right),1,0}\right)=b_{\left(k,i\right),1,j+1}$
pour $0\leq j\leq\delta_{\left(k,i\right)}-2$ et si on pose
$T_{\left(k,i\right)}\left(X\right)=X^{\delta_{\left(k,i\right)}}+\sum_{p=0}^{\delta_{\left(k,i\right)}-1}\lambda_{p}X^{p}$
on obtient $\bigl(b_{\left(k,i\right),1,\delta_{\left(k,i\right)}-1}\bigr)^{2}=V^{\delta_{\left(k,i\right)}}\left(b_{\left(k,i\right),1,0}\right)=\sum_{p=0}^{\delta_{\left(k,i\right)}-1}\lambda_{p}V^{p}\left(b_{\left(k,i\right),1,0}\right)=\sum_{p=0}^{\delta_{\left(k,i\right)}-1}\lambda_{p}b_{\left(k,i\right),1,p}$,
par conséquent $\bigl(b_{\left(k,i\right),1,\delta_{\left(k,i\right)}-1}\bigr)^{2}=C_{T_{\left(k,i\right)}}\bigl(b_{\left(k,i\right),1,\delta_{\left(k,i\right)}-1}\bigr)$.
Pour chaque entier $1\leq k\leq s$ et $1\leq i\leq r_{k}$,
notons $S_{\left(k,i\right),1}$ le sous-espace engendré par
le système $\left\{ b_{\left(k,i\right),1,j};0\leq j<\delta_{\left(k,i\right)}\right\} $,
si $S_{\left(k,i\right),1}\neq E_{\left(k,i\right)}$ il existe
un élément $b_{\left(k,i\right),2,0}\neq0$ dans un supplémentaire
de $S_{\left(k,i\right),1}$ à $E_{\left(k,i\right)}$ à partir
duquel on construit comme on vient de le faire le système libre
$\left\{ b_{\left(k,i\right),2,i};0\leq j<\delta_{\left(k,i\right)}\right\} $.
En poursuivant ainsi on obtient $q_{\left(k,i\right)}$ éléments
$\left\{ b_{\left(k,i\right),j,0};1\leq j\leq q_{\left(k,i\right)}\right\} $
de $E_{\left(k,i\right)}$ qui génèrent la base $\bigcup_{j=1}^{q_{\left(k,i\right)}}\left\{ b_{\left(k,i\right),j,p}/0\leq p<\delta_{\left(k,i\right)}\right\} $
de $E_{\left(k,i\right)}$ dont les éléments vérifient la table
de multiplication donnée dans la proposition. 

Finalement, il résulte de la proposition \ref{prop:C_Di=000026C_D}
que l'algèbre $A\left(\mathbf{v},\mathbf{q}\right)$ obtenue
ci-dessus vérifie le train polynôme $T$.
\end{proof}
\medskip{}

\section{Algèbres d'évolution sur un corps fini de caractéristique 2}

Dans ce qui suit nous allons voir que les algèbres d'évolution
sur le corps $\mathbb{F}_{2}$ sont dans certains cas nilpotentes
et dans tous les cas ultimement périodiques et train plénières.\medskip{}

Les algèbres d'évolution ont été étudiées et popularisées sous
ce nom par J.P. Tian \cite{Tian-08}. Elles ont été introduites
et utilisées pour la première fois par Etherington {[}\cite{Ether-41},
p. 34{]} dans l'étude de la dynamique d'un système diallélique
dans le cas de l'autofécondation stricte et en l'absence de mutation
génétique.\medskip{}

\begin{defn}
Une $\mathbb{K}$-algèbre $A$ est d'\emph{évolution} s'il existe
une base $\left(e_{i}\right)_{i\in I}$ vérifiant 
\[
e_{i}^{2}=\sum_{j\in I}\alpha_{ji}e_{j}\;\text{et}\;e_{i}e_{j}=0\text{ si }i\neq j.
\]

Une base de $A$ vérifiant cette propriété est appelée une \emph{base
naturelle}.\medskip{}
\end{defn}

En introduisant le morphisme d'espace vectoriel $S:A\rightarrow A$
défini sur la base naturelle $\left(e_{i}\right)_{i\in I}$ par
$S\left(e_{i}\right)=\sum_{j\in I}\alpha_{ji}e_{j}$, on a $e_{i}^{2}=S\left(e_{i}\right)$
et l'opérateur d'évolution sur $A$ s'écrit pour tout $x\in A$,
$x=\sum_{i\in I}\lambda_{i}e_{i}$ sous la forme
\begin{equation}
V\left(x\right)=\sum_{i\in I}\lambda_{i}^{2}S\left(e_{i}\right).\label{eq:V=000026S}
\end{equation}

En particulier si $\mathbb{K}=\mathbb{F}_{2}$ , on montre aisément
que 
\begin{equation}
V^{k}\left(x\right)=\sum_{i\in I}\lambda_{i}S^{k}\left(e_{i}\right),\quad\left(k\geq1\right).\label{eq:V=000026S,q=00003D1}
\end{equation}

En dimension finie, la matrice de l'endomorphisme $S$ dans la
base naturelle est appelée la \emph{matrice des constantes de
structure} de l'algèbre d'évolution relativement à la base naturelle.
\begin{prop}
Soient $A$ une $\mathbb{F}_{2}$-algèbre d'évolution de dimension
finie et $S$ la matrice des constantes de structure relativement
à une base naturelle de $A$. Concernant l'opérateur d'évolution
$V$ de $A$, on a 

(a) $V$ est nilpotent si et seulement si $S$ est nilpotente.

(b) $V$ est ultimement périodique.

(c) $V$ est train plénier de degré égal au degré du polynôme
minimal de $S$.
\end{prop}

\begin{proof}
(\emph{a}) Cela résulte immédiatement de la relation (\ref{eq:V=000026S,q=00003D1}).

Soit $d$ la dimension de $A$. 

(\emph{b}) Montrons que l'opérateur $V$ est ultimement périodique.
Pour cela on considère l'ensemble $\left\{ S^{k};k\geq1\right\} $,
cet ensemble est fini de cardinal au plus $2^{d^{2}}$, par conséquent
il existe deux entiers $n<m$ tels que $S^{n}=S^{m}$ alors de
(\ref{eq:V=000026S,q=00003D1}) on a $V^{n+\left(m-n\right)}=V^{n}$.
Il en résulte que l'ensemble $\left\{ \left(s,r\right);V^{r+s}=V^{r}\right\} $
n'est pas vide, il possède un plus petit élément $\left(p,n\right)$
pour l'ordre lexicographique et donc $V$ est $\left(n,p\right)$-périodique.

(\emph{c}) L'opérateur $V$ est train plénier. En effet, si $T$
est le polynôme minimal de $S$, de $T\left(S\right)=0$ et de
(\ref{eq:V=000026S,q=00003D1}) on déduit que $T\left(V\right)=0$.
\end{proof}
Dans le cas $\mathbb{K}=\mathbb{F}_{2^{q}}$ ($q\geq2)$, la
situation est moins simple. 

A partir de la relation (\ref{eq:V=000026S}) on obtient $V^{2}\left(x\right)=\sum_{k\in I}\left(\sum_{i,j\in I}\alpha_{kj}\alpha_{ji}^{2}\lambda_{i}^{4}\right)e_{k}$
et par récurrence on a: 
\[
V^{k+1}\left(x\right)=\sum_{j\in I}\Bigl(\sum_{i_{0},\ldots,i_{k}\in I}\alpha_{ji_{k}}\alpha_{i_{k}i_{k-1}}^{2}\ldots\alpha_{i_{1}i_{0}}^{2^{k+1}}\lambda_{i_{0}}^{2^{k+1}}\Bigr)e_{j},\quad\left(k\geq1\right).
\]
En dimension finie ceci s'écrit plus simplement en utilisant
le produit d'Hadamard des matrices: $\left(a_{ij}\right)_{i,j}\odot\left(b_{ij}\right)_{i,j}=\left(a_{ij}b_{ij}\right)_{i,j}$
, sous la forme 
\[
V^{k+1}\left(x\right)=\sum_{j=1}^{d}\lambda_{j}^{2^{k+1}}\left(SS^{\odot2}S^{\odot4}\cdots S^{\odot2^{k+1}}\right)\left(e_{j}\right).
\]

Avec ceci on a immédiatement le résultat suivant.
\begin{prop}
Soient $A$ une algèbre d'évolution sur le corps $\mathbb{F}_{2^{q}}$
($q\geq2)$ de dimension finie et $S$ la matrice des constantes
de structure relativement à une base naturelle de $A$. L'opérateur
d'évolution $V$ de $A$ est nilpotent si et seulement si il
existe un entier $k\geq1$ tel que $SS^{\odot2}S^{\odot4}\cdots S^{\odot2^{k+1}}=0$.
\end{prop}

Et avec les propositions \ref{prop:Dim_finie=000026UP} et \ref{prop:dim_finie=000026TP}
on a:
\begin{prop}
Soient $A$ une algèbre d'évolution sur le corps $\mathbb{F}_{2^{q}}$
($q\geq2)$ de dimension finie et $S$ la matrice des constantes
de structure relativement à une base naturelle de $A$. L'opérateur
d'évolution $V$ de $A$ est ultimement périodique de période
$p$ diviseur de $q$ et train plénier.
\end{prop}

\begin{proof}
Montrons le résultat concernant la période de l'opérateur $V$.
Si l'opérateur $V$ est nilpotent de degré $n$ on sait que dans
ce cas $V$ est $\left(n,1\right)$-périodique. Supposons que
$V$ ne soit pas nilpotent, soit $\left(e_{1},\ldots,e_{d}\right)$
une base naturelle de $A$, il existe existe un $1\leq i\leq d$
tel que $V^{k}\left(e_{i}\right)\neq0$, quel que soit $k\geq1$.
D'après la proposition \ref{prop:Dim_finie=000026UP} il existe
un couple minimal $\left(p,n\right)$ tel que $V$ est $\left(n,p\right)$-périodique,
alors pour $\lambda\in\mathbb{F}_{2^{q}}$ de $V^{n+p}\left(e_{i}\right)=V^{n}\left(e_{i}\right)$
et de $V^{n+p}\left(\lambda e_{i}\right)=V^{n}\left(\lambda e_{i}\right)$
on déduit que $\left(\lambda^{2^{n+p}}-\lambda^{2^{n}}\right)V\left(e_{i}\right)=0$
par conséquent $\lambda^{2^{p}}-\lambda=0$ d'où l'on déduit
que $p$ divise $q$.
\end{proof}

\section{Algèbres pondérées quasi-constantes, ultimement périodiques
et train plénières }

Usuellement, une algèbre pondérée $\left(A,\omega\right)$ est
la donnée d'une $\mathbb{K}$-algèbre $A$ et d'un morphisme
non nul d'algèbres $\omega:A\rightarrow\mathbb{K}$ appelé une
pondération de $A$. Pour une pondération $\omega$ de $A$ donnée,
l'image $\omega\left(x\right)$ d'un élément $x\in A$ est appelé
le poids de $x$ et on note $H_{\omega}=\left\{ x\in A;\omega\left(x\right)=1\right\} $
l'hyperplan affine des éléments de poids 1 de $A$, comme $\omega\neq0$
on a $H_{\omega}\neq\textrm{Ø}$. 

\medskip{}

Il est immédiat qu'il n'existe pas d'algèbre nil-plénière pondérée. 

La pondération permet d'étudier des opérateurs d'évolution vérifiant
des identités polynomiales sans faire d'hypothèse sur le corps
$\mathbb{F}$ comme dans les théorèmes \ref{thm:Existence-QC},
\ref{thm:Existence _des_UP} et \ref{thm:Existence_TP}.

\subsection{Algèbres pondérées quasi-constantes }

\textcompwordmark{}

\medskip{}

\begin{defn}
Soit $\left(A,\omega\right)$ une $\mathbb{F}$-algèbre pondérée.
On dit que $A$ est quasi-constante de degré $n$ s'il existe
$e\in A$, $e\neq0$ tel que $V^{n}\left(x\right)=\omega\left(x\right)^{2^{n}}e$
pour tout $x\in A$, avec $n\in\mathbb{N^{*}}$ minimal. On désigne
par $\left(A,\omega,e\right)$ une telle algèbre. 
\end{defn}

\begin{prop}
\label{prop:Poids1} Pour une $\mathbb{F}$-algèbre pondérée
$\left(A,\omega\right)$, les énoncés suivants sont équivalents:

(i) il existe $e\in A$ tel que $\left(A,\omega,e\right)$ est
une algèbre quasi-constante de degré $n$;

(ii) il existe $e\in A$ tel que $V^{n}\left(x\right)=e$ pour
tout $x\in H_{\omega}$ et $V^{n-1}\left(y\right)\neq e$ pour
un $y\in H_{\omega}$;

(iii) il existe $e\in H_{\omega}$ tel que $e^{2}=e$, $V^{n}\left(z\right)=0$
pour tout $z\in\ker\omega$ et $V^{n-1}\left(z'\right)\neq0$
pour un $z'\in\ker\omega$.
\end{prop}

\begin{proof}
$\left(i\right)\Rightarrow\left(ii\right)$ Immédiat.

$\left(ii\right)\Rightarrow\left(iii\right)$ Soit $e\in A$
tel que $V^{n}\left(x\right)=e$ pour tout $x\in H_{\omega}$.
On a $\omega\left(e\right)\neq0$ sinon on aurait $\omega\left(V^{n}\left(x\right)\right)=0$
d'où $\omega\left(x\right)=0$ pour tout $x\in A$. On a $\omega\left(e\right)^{-1}e\in H_{\omega}$
donc $V^{n}\bigl(\omega\left(e\right)^{-1}e\bigr)=e$ ou $V^{n}\left(e\right)=\omega\left(e\right)^{2^{n}}e$,
en appliquant la pondération $\omega$ à cette relation on obtient
$\omega\left(e\right)^{2^{n}}\left(\omega\left(e\right)-1\right)=0$
d'où $\omega\left(e\right)=1$. Ensuite on a: $e^{2}=V\left(e\right)=V\left(V^{n}\left(e\right)\right)=V^{n}\left(V\left(e\right)\right)=\omega\left(e^{2}\right)^{2^{n}}e=\omega\left(e\right)^{2^{n+1}}e=e$.
Pour tout $z\in\ker\omega$ on a $e+z\in H_{\omega}$ donc $V^{n}\left(e+z\right)=e$,
or $V^{n}\left(e+z\right)=V^{n}\left(e\right)+V^{n}\left(z\right)=e+V^{n}\left(z\right)$
par conséquent $V^{n}\left(z\right)=0$. Il existe $y\in H_{\omega}$
tel que $V^{n-1}\left(y\right)\neq e$, il existe $z'\in\ker\omega$
tel que $y=e+z'$ on a $V^{n-1}\left(y\right)=e+V^{n-1}\left(z'\right)$
et donc nécessairement $V^{n-1}\left(z'\right)\neq0$.

$\left(iii\right)\Rightarrow\left(i\right)$ Pour tout $\alpha\in\mathbb{F}$
et $z\in\ker\omega$ on a $\omega\left(\alpha e+z\right)=\alpha$
et $V^{n}\left(\alpha e+z\right)=V^{n}\left(\alpha e\right)+V^{n}\left(z\right)=\alpha^{2^{n}}e$,
l'entier $n$ vérifiant cette identité est minimal car s'il existe
$m<n$ tel que $V^{m}\left(x\right)=\omega\left(x\right)^{2^{m}}e$
pour tout $x\in A$, alors on a $V^{m}\left(z\right)=0$ pour
tout $z\in\ker\omega$, contradiction.
\end{proof}
La définition d'une algèbre quasi-constante $A$ dépend de la
donnée d'une pondération $\omega$ et d'un élément $e\in A$.
\begin{prop}
La pondération $\omega$ et l'élément $e$ d'une $\mathbb{F}$-algèbre
quasi-constante $\left(A,\omega,e\right)$ sont uniques.
\end{prop}

\begin{proof}
Soit $\left(A,\omega,e\right)$ une $\mathbb{F}$-algèbre quasi-constante
d'ordre $n$. Supposons qu'il existe une pondération $\omega'$
de $A$ et $e'\in A$ tels que $\left(A,\omega',e'\right)$ soit
quasi-constante de degré $n$. De $V^{n}\left(x\right)=\omega'\left(x\right)^{2^{n}}e'$
et $V^{n}\left(x\right)=\omega\left(x\right)^{2^{n}}e$ il vient
$\omega'\left(x\right)^{2^{n}}e'=\omega\left(x\right)^{2^{n}}e\;\left(\star\right)$
pour tout $x\in A$, comme $e',e\neq0$ on en déduit que $\ker\omega'=\ker\omega$,
ceci ajouté au fait que d'après le résultat (\emph{iii}) de la
proposition \ref{prop:Poids1} on a $\omega\left(e\right)=\omega'\left(e'\right)=1$,
permet d'affirmer que $\omega'=\omega$. Ensuite en prenant $x=e'$
dans $\left(\star\right)$ et compte tenu du (\emph{iii}) de
la proposition \ref{prop:Poids1} que $\omega'\left(e'\right)=1$,
on a $e'=\omega\left(e'\right)^{2^{n}}e$ et comme $\omega\left(e\right)=1$
on en déduit que $\omega\left(e'\right)=\omega\left(e'\right)^{2^{n}}$
par conséquent $e'=\omega\left(e'\right)e$. Alors avec (\emph{iii})
de la proposition \ref{prop:Poids1} on a $e'=e'^{2}=\omega\left(e'\right)^{2}e^{2}=\omega\left(e'\right)^{2}e=\omega\left(e'\right)e'$
d'où $\omega\left(e'\right)=1$ et donc $e'=e$. 
\end{proof}
\begin{prop}
Une $\mathbb{F}$-algèbre pondérée $\left(A,\omega,e\right)$
de dimension $d\geq2$ quasi-constante de degré $n$ est semi-isomorphe
à une $\mathbb{F}$-algèbre, notée $A\left(\mathbf{s}\right)$,
définie par la donnée:

– d'un $m$-uplet d'entiers $\mathbf{s}=\left(s_{1},\ldots,s_{m}\right)$
où $m\geq1$, $n=s_{1}\geq s_{2}\geq\ldots\geq s_{m}\geq1$ et
$s_{1}+s_{2}+\cdots+s_{m}=d-1$; 

\smallskip{}

– d'une base $\left\{ e\right\} \cup\bigcup_{i=1}^{m}\left\{ e_{i,j};1\leq j\leq s_{i}\right\} $
telle que 

\hspace{13mm} $\omega\left(e\right)=1$,\smallskip{}

\hspace{13mm} $\ker\omega$ a pour base $\bigcup_{i=1}^{m}\left\{ e_{i,j};1\leq j\leq s_{i}\right\} $,
\begin{align*}
e^{2} & =e,\qquad e_{i,j}^{2}=\begin{cases}
e_{i,j+1} & \text{si }1\leq j\leq s_{i}-1\\
0 & \text{si }j=i
\end{cases},\quad\left(1\leq i\leq m\right)
\end{align*}
 et les autres produits étant définis arbitrairement sur $\ker\omega$.
\end{prop}

\begin{proof}
Soit $\left(A,\omega,e\right)$ une $\mathbb{F}$-algèbre pondérée
de dimension $d$ quasi-constante de degré $n$. D'après la proposition
\ref{prop:Poids1} on a $\omega\left(e\right)=1$ et donc $A=\mathbb{F}e\oplus\ker\omega$,
d'après l'assertion $\left(iii\right)$ de la proposition \ref{prop:Poids1}
l'idéal $\ker\omega$ est nil-plénier de degré $n$ par conséquent
d'après la proposition \ref{prop:Base_Alg_Res} il est semi-isomorphe
à une algèbre de type $\left(\ker\omega\right)\left(\mathbf{s}\right)$. 
\end{proof}
\medskip{}

\subsection{Algèbres pondérées ultimement périodiques (ou algèbres de Bernstein
périodiques)}
\begin{defn}
\cite{RV-94} Une $\mathbb{F}$-algèbre pondérée $\left(A,\omega\right)$
est de Bernstein d'ordre $n$ et de période $p$, en abrégé $B\left(n,p\right)$-algèbre,
si elle vérifie 
\[
V^{n+p}\left(x\right)=\omega\left(x\right)^{2^{n}\left(2^{p}-1\right)}V^{n}\left(x\right),
\]
pour tout $x\in A$ avec $\left(p,n\right)\in\mathbb{N}^{*}\times\mathbb{N}$
minimal pour l'ordre lexicographique. 

Et de manière générale, on dit que $\left(A,\omega\right)$ est
de Bernstein périodique s'il existe deux entiers $n$ et $p$
tels que $A$ soit une $B\left(n,p\right)$-algèbre.
\end{defn}

\begin{thm}
\label{thm:BPw=000026QC}Une $\mathbb{F}$-algèbre pondérée est
de Bernstein périodique si et seulement si elle est quasi-constante.
\end{thm}

\begin{proof}
Pour la condition nécessaire, soient $\left(A,\omega\right)$
une $B\left(n,p\right)$-algèbre, il existe des entiers $a\geq0$
et $q\geq1$ impair tels que $p=2^{a}q$. Soit $x\in H_{\omega}$
fixé, on pose $e=\sum_{k=0}^{q-1}V^{n+2^{a}k}\left(x\right)$,
on a $\omega\left(e\right)=q$ donc $\omega\left(e\right)\in\mathbb{N}$
et $V^{2^{a}}\left(e\right)=\sum_{k=1}^{q}V^{n+2^{a}k}\left(x\right)=V^{n+p}\left(x\right)+\sum_{k=1}^{q-1}V^{n+2^{a}k}\left(x\right)=e$,
en appliquant la forme $\omega$ à ce résultat on obtient $\omega\left(e\right)^{2^{2^{a}}}=\omega\left(e\right)$,
mais le polynôme $X^{2^{2^{a}}}-X$ n'a que deux racines dans
$\mathbb{N}$ qui sont $0$ et $1$, par conséquent on a $\omega\left(e\right)=1$.
Ensuite on a $V^{2^{a}\left(q+n\right)}\left(e\right)=e$ et
$V^{2^{a}\left(q+n\right)}\left(z\right)=V^{p+2^{a}n}\left(z\right)=V^{\left(2^{a}-1\right)n}V^{n+p}\left(z\right)=0$
pour tout $z\in\ker\omega$, on en déduit que $V^{2^{a}\left(q+n\right)}\left(\alpha e+z\right)=\alpha^{2^{2^{a}\left(q+n\right)}}e$,
autrement dit $V^{2^{a}\left(q+n\right)}\left(x\right)=\omega\left(x\right)^{2^{2^{a}\left(q+n\right)}}e$
pour tout $x\in A$, il en résulte que l'ensemble $\left\{ k\in\mathbb{N};V^{k}\left(x\right)=\omega\left(x\right)^{2^{k}}e,\forall x\in A\right\} $
n'est pas vide, il admet un plus petit élément $r\leq2^{a}\left(q+n\right)$
pour lequel l'algèbre $A$ est quasi-constante de degré $r$. 

\medskip{}

Pour la condition suffisante, soit $A$ une $\mathbb{F}$-algèbre
quasi-constante de degré $n$, pour tout $x\in A$ on a $V^{n}\left(x\right)=\omega\left(x\right)^{2^{n}}e\;\left(\star\right)$
alors $V^{n+1}\left(x\right)=\bigl(\omega\left(x\right)^{2^{n}}\bigr)^{2}V\left(e\right)=\omega\left(x\right)^{2^{n+1}}e=\omega\left(x\right)^{2^{n}}V^{n}\left(x\right)$
car $e^{2}=e$. Montrons par l'absurde que le couple $\left(1,n\right)$
est minimal pour l'ordre lexicographique, si l'on suppose qu'il
existe $m<n$ tel que $V^{m+1}\left(x\right)=\omega\left(x\right)^{2^{m}}V^{m}\left(x\right)$
pour tout $x\in A$ alors on aurait $V^{m+1}\left(z\right)=0$
pour tout $z\in\ker\omega$, par conséquent d'après la proposition
\ref{prop:Poids1} on aurait $m+1=n$ et donc $V^{n}\left(x\right)=\omega\left(x\right)^{2^{n-1}}V^{n-1}\left(x\right)$,
ceci joint à la relation $\left(\star\right)$ donne $\omega\left(x\right)^{2^{n-1}}V^{n-1}\left(x\right)=\omega\left(x\right)^{2^{n}}e$,
on en déduit que pour tout $x\in H_{\omega}$ on aurait $V^{n-1}\left(x\right)=e$,
ce qui d'après le (\emph{ii}) de la proposition \ref{prop:Poids1}
est impossible.
\end{proof}
\medskip{}

\subsection{Algèbres pondérées train plénières}
\begin{defn}
Etant donnée une $\mathbb{F}$-algèbre pondérée $\left(A,\omega\right)$,
on dit que $A$ vérifie une identité train plénière de degré
$n$ ($n\geq1$) s'il existe $\left(\alpha_{0},\ldots,\alpha_{n-1}\right)\in\mathbb{K}^{n}\setminus\left(0,\ldots,0\right)$
tel que pour tout $x\in A$ on ait:
\[
V^{n}\left(x\right)+\sum_{k=0}^{n-1}\alpha_{i}\omega\left(x\right)^{2^{n}-2^{k}}V^{k}\left(x\right)=0.
\]

Une $\mathbb{F}$-algèbre pondérée $\left(A,\omega\right)$ est
train plénière de degré $n$ si elle vérifie une identité train
plénière de degré $n$ et si elle ne vérifie pas d'identité train
plénière de degré $<n$.

En appliquant la forme $\omega$ à l'identité train plénière
pour $x\in H_{\omega}$ on obtient que $\sum_{k=0}^{n-1}\alpha_{k}+1=0$.
\end{defn}

\begin{thm}
Une $\mathbb{F}$-algèbre pondérée est une train algèbre plénière
si et seulement si elle est quasi-constante.
\end{thm}

\begin{proof}
Soit $\left(A,\omega\right)$ une algèbre train plénière de degré
$n$. Si $A$ est de dimension 1 le résultat est trivial, on
suppose donc que $A$ est de dimension $\geq2$, l'idéal $\ker\omega$
est nil-plénier, soit $p\leq n$ son degré. Soit $a\in H_{\omega}$
donné, pour tout $x\in H_{\omega}$ tel que $x\neq a$ on a $a-x\in\ker\omega$
donc $V^{p}\left(a-x\right)=0$, il en résulte que $V^{p}\left(x\right)=V^{p}\left(a\right)$,
donc si on pose $e=V^{p}\left(a\right)$ on a $V^{p}\left(x\right)=e$
pour tout $x\in H_{\omega}$, on en déduit que $V^{p}\left(x\right)=\omega\left(x\right)^{2^{p}}e$
quel que soit $x\in A$ tel que $\omega\left(x\right)\neq0$
et comme $V^{p}\left(z\right)=0$ pour tout $z\in\ker\omega$,
par additivité de $V$ on a finalement $V^{p}\left(x\right)=\omega\left(x\right)^{2^{p}}e$
pour tout $x\in A$. La réciproque découle du théorème \ref{thm:BPw=000026QC}.
\end{proof}
\bigskip{}

\end{document}